\documentclass[11pt]{amsart}

\usepackage{epigamath-voisin-runin}

\usepackage[english]{babel}

\numberwithin{equation}{section}

\usepackage[shortlabels]{enumitem}
\setlist[enumerate]{label=\textup{(\roman*)}}

\usepackage[english]{babel}
\usepackage{amsmath,amsthm,amssymb,bm,doi,enumitem,mathrsfs,mathtools,thmtools,titlesec,url}
\usepackage[T1]{fontenc}

\usepackage{tikz}
\usetikzlibrary{arrows,calc,decorations,intersections,matrix}
\usepackage{tikz-cd}
\tikzset{
    labl/.style={anchor=south, rotate=90, inner sep=.5mm} 
}
\tikzset{
  symbol/.style={
    draw=none,
    every to/.append style={
      edge node={node [sloped, allow upside down, auto=false]{$#1$}}}
  }
}

\theoremstyle{plain}
\newtheorem*{theorem*}{Theorem}
\newtheorem{theorem}{Theorem}[section]
\newtheorem{proposition}[theorem]{Proposition}
\newtheorem{lemma}[theorem]{Lemma}

\theoremstyle{definition}
\newtheorem{definition}[theorem]{Definition}

\theoremstyle{remark}
\newtheorem{remark}[theorem]{Remark}
\newtheorem{remarks}[theorem]{Remarks}

\newcommand{\CC}{\mathbb{C}}

\newcommand{\PP}{\mathbb{P}}
\newcommand{\QQ}{\mathbb{Q}}

\newcommand{\ZZ}{\mathbb{Z}}

\newcommand{\scrE}{\mathscr{E}}
\newcommand{\scrH}{\mathscr{H}}
\newcommand{\scrL}{\mathscr{L}}
\newcommand{\scrO}{\mathscr{O}}
\newcommand{\scrU}{\mathscr{U}}

\newcommand{\frh}{\mathfrak{h}}

\def\lowsim{\vbox to 0pt{\vss\hbox{$\scriptstyle\sim$}\vskip-2.5pt}}
\newcommand{\isomarrow}{\xrightarrow{\lowsim}}
\newcommand{\isomlongarrow}{\xrightarrow{\;\lowsim\;}}
\let\tto=\longrightarrow

\DeclareMathOperator{\A}{\mathfrak{A}}
\DeclareMathOperator{\ab}{ab}

\DeclareMathOperator{\alg}{alg}
\DeclareMathOperator{\Betti}{B}
\DeclareMathOperator{\CGr}{\mathsf{CGr}}
\DeclareMathOperator{\charact}{char}
\DeclareMathOperator{\CH}{CH}
\DeclareMathOperator{\ch}{ch}
\DeclareMathOperator{\class}{cl}
\DeclareMathOperator{\Db}{\mathrm{D}^\mathrm{b}}

\DeclareMathOperator{\Gal}{Gal}
\DeclareMathOperator{\GL}{GL}
\DeclareMathOperator{\Grass}{\mathsf{Gr}}
\DeclareMathOperator{\Griff}{\mathsf{Griff}}
\DeclareMathOperator{\Hom}{Hom}
\DeclareMathOperator{\Hilb}{Hilb}
\DeclareMathOperator{\id}{id}
\DeclareMathOperator{\Ker}{Ker}
\DeclareMathOperator{\Ku}{Ku}

\DeclareMathOperator{\nr}{nr}
\DeclareMathOperator{\num}{num}

\DeclareMathOperator{\Pic}{Pic}
\DeclareMathOperator{\pr}{pr}

\DeclareMathOperator{\pt}{pt}

\DeclareMathOperator{\Spec}{Spec}
\DeclareMathOperator{\Sym}{Sym}

\DeclareMathOperator{\topo}{top}
\DeclareMathOperator{\td}{td}
\DeclareMathOperator{\tr}{tr}

\newcommand{\motH}{\mathbf{H}}

\newcommand{\unitmot}{\bm{1}}

\newcommand{\CHM}{\mathsf{CHM}}
\newcommand{\Mot}{\mathsf{Mot}}

\newcommand{\Fbar}{\overline{F}}
\newcommand{\Kbar}{\overline{K}}
\newcommand{\pprime}{{\prime\prime}}

\YearArticle{2024} \EpigaArticleNr{17} \ReceivedOn{July 20, 2022}
\InFinalFormOn{January 12, 2024}
\AcceptedOn{January 29, 2024}

\title{Algebraic cycles on Gushel--Mukai varieties}
\titlemark{Algebraic cycles on Gushel--Mukai varieties}

\author{Lie Fu}
\address{Universit\'e de Strasbourg, IRMA, Strasbourg, France}
\email{lie.fu@math.unistra.fr}
\author{Ben Moonen}
\address{Radboud University Nijmegen, IMAPP, Nijmegen, The Netherlands}
\email{b.moonen@science.ru.nl}

\authormark{L.~Fu and B.~Moonen}

\AbstractInEnglish{We study algebraic cycles on complex Gushel--Mukai (GM) varieties. We prove the generalised Hodge conjecture, the (motivated) Mumford--Tate conjecture, and the generalised Tate conjecture for all GM varieties. We compute all integral Chow groups of GM varieties, except for the only two infinite-dimensional cases ($1$-cycles on GM fourfolds and $2$-cycles on GM sixfolds). We prove that if two GM varieties are generalised partners or generalised duals, their rational Chow motives in middle degree are isomorphic.}



\acknowledgement{L.F.\ is supported by the Radboud Excellence Initiative, by the University of Strasbourg Institute for Advanced Study (USIAS), and by the Agence Nationale de la Recherche (ANR) under project ANR-20-CE40-0023.} 

\dedication{Dedicated to Claire Voisin, with admiration}

\begin{document}

%%%%%%%%%%%%%%%%%%%%%%%%%%%%%%%
% Title page
%%%%%%%%%%%%%%%%%%%%%%%%%%%%%%%

%\removeabove{}
%\removebetween{}
%\removebelow{}

\maketitle

\begin{prelims}

\DisplayAbstractInEnglish

\bigskip

\DisplayKeyWords

\medskip

\DisplayMSCclass

%\bigskip

%\languagesection{Fran\c{c}ais}

%\bigskip

%\DisplayTitleInFrench

%\medskip

%\DisplayAbstractInFrench

\end{prelims}

%%%%%%%%%%%%%%%%%%%%%
% Table of Contents
%%%%%%%%%%%%%%%%%%%%%

\newpage

\setcounter{tocdepth}{1}

\tableofcontents

%%%%%%%%%%%%%%%%%%%%%
% Content begins here
%%%%%%%%%%%%%%%%%%%%%

\section{Introduction}

\subsection{}
The main goal of this paper and its companion~\cite{FuMoonen-TCGMV} is to prove several results about algebraic cycles on Gushel--Mukai (GM) varieties. The focus of the present paper is on GM varieties in characteristic~$0$. Several results that we obtain here are used in~\cite{FuMoonen-TCGMV}, in which we prove the Tate conjecture for even-dimensional GM varieties in characteristic $p\geq 5$.

GM varieties form a class of Fano varieties that admit a simple explicit definition and that are interesting because of their rich geometry and their connections to hyperk\"ahler varieties. We refer to the series of papers by Debarre and Kuznetsov (see \cite{DK-GMClassification, DK-Kyoto, DK-DoubleCov, DK-GMJacobian}) for an in-depth study. A nice starting point is Debarre's overview paper~\cite{Debarre-GM}. 

We work over an algebraically closed field~$K$ of characteristic~$0$. For the purposes of this paper, only smooth GM varieties of dimension greater than~$2$ are of interest. These exist in dimensions $n \in \{3,4,5,6\}$; they can be realised as intersections
\[
X = \CGr(2,V_5) \cap \PP(W) \cap Q
\]
of the cone $\CGr(2,V_5) \subset \PP(K \oplus \wedge^2 V_5)$ over the Grassmannian of $2$-planes in a $5$-dimensional vector space~$V_5$ with a linear subspace $\PP(W) \subset \PP(K \oplus \wedge^2 V_5)$ of dimension~$n+4$ and a quadric~$Q$.

\subsection{}
The cohomology of an $n$-dimensional GM variety~$X$ is purely of Tate type, except in middle degree. If $n$ is odd, $H^n(X)$ corresponds to a $10$-dimensional abelian variety, the intermediate Jacobian. If $n$ is even, $H^n(X)$ is of K3 type, with Hodge numbers $1$-$22$-$1$. A first main theme in this paper is that  almost all Chow groups also have a relatively simple structure. 

In Section~\ref{sec:ChowGroupGM} we compute all Chow groups of complex GM varieties with \emph{integral} coefficients, except for the two infinite-dimensional cases, namely $1$-cycles on GM fourfolds and $2$-cycles on GM sixfolds. For GM fivefolds, the result is due to Zhou in her thesis~\cite{ZhouLin}. In dimensions $3$ and~$4$, the results are easily deduced from the results of Bloch and Srinivas, see~\cite{BlochSrinivas}, and Laterveer, see~\cite{Laterveer-SmallChow}. Similar results with $\QQ$-coefficients were obtained by Laterveer using different methods.

Our main new contribution concerns $1$-cycles and $3$-cycles on GM sixfolds~$X$. We prove that the Fano variety of lines on~$X$ is rationally chain connected, so that all lines on~$X$ have the same class in~$\CH^5(X)$. By a nice geometric argument using ``successive ruled surfaces'' as in the work~\cite{TianZong} of Tian and Zong, we show that every $1$-cycle on~$X$ is equivalent to an integral multiple of the class of a line, so that we obtain $\CH^5(X) \cong \ZZ$. 

For cycles of codimension~$3$ on a GM sixfold, we show that $\CH^3(X)_{\hom} = 0$. The main step is to show that the Griffiths group~$\Griff^3(X)$ is zero. As the Hodge conjecture with $\QQ$-coefficients for~$X$ is true (see below), the image of $\CH^3(X)$ in cohomology is of finite index in the space of Hodge classes in $H^6(X,\ZZ)\left(3\right)$.

\subsection{}
A second main result of the paper concerns the (generalised) Hodge conjecture and the Tate conjecture, and, as a bridge between them, the Mumford--Tate conjecture. (See Section~\ref{sec:MTC+GTC} for the formulation of these conjectures.) The result we obtain is the following.

\begin{theorem*}
Let $X$ be a complex GM variety. 
\begin{enumerate}
\item The generalised Hodge conjecture for $X$ is true.

\item If\, $\dim(X)$ is even, the Mumford--Tate conjecture for $X$ is true.

\item If\, $\dim(X)$ is even, the generalised Tate conjecture for $X$ is true.
\end{enumerate}  
\end{theorem*}

Most cases of the generalised Hodge conjecture (GHC) are in fact covered by work of Laterveer, and the remaining case of GM sixfolds is easily deduced from our calculations of Chow groups. 

Our proof of the Mumford--Tate conjecture (MTC) is based on the work of Andr\'e~\cite{Andre-ShafTate}. It is crucial here that the middle cohomology of~$X$ is of K3 type, which is why we have to assume that $\dim(X)$ is even. (For GM varieties of odd dimension, the Mumford--Tate conjecture is not known; in this case it is a problem about $10$-dimensional abelian varieties, which is a much more difficult case to handle.) The convenience of Andr\'e's results is that we only need to verify a couple of conditions, the most important of which is that $X$ should appear as a fibre in a smooth family whose image under the period map contains an open subset of the appropriate period domain. This is known to be true for GM varieties by the work of Debarre and Kuznetsov. 

The generalised Tate conjecture, finally, is a formal consequence of the GHC and the MTC.

The above theorem is an important ingredient for our proof of the Tate conjecture for GM varieties in characteristic $p\geq 5$ in~\cite{FuMoonen-TCGMV}.

\subsection{}
In the final section of the paper, we turn to Chow motives of GM varieties, and we prove a result about so-called generalised partners or generalised duals. A central theme in the study of GM varieties is that a lot of important information about them can be encoded in terms of multilinear algebra data. In particular, to a GM variety over an algebraically closed field~$K$ (with $\charact(K)=0$), one can associate a ``Lagrangian data set'', which is a triple $(V_6,V_5,A)$ consisting of a $6$-dimensional $K$-vector space~$V_6$, a hyperplane $V_5 \subset V_6$, and a Lagrangian subspace $A \subset \wedge^3 V_6$ with respect to the natural symplectic form $\wedge^3 V_6 \times \wedge^3 V_6 \to \det(V_6)$. GM varieties $X$ and~$X^\prime$ whose dimensions have the same parity are said to be generalised partners (resp.\ generalised duals) if there exists an isomorphism $f\colon V_6(X) \isomarrow V_6(X^\prime)$ such that the induced isomorphism $\wedge^3 V_6(X) \isomarrow \wedge^3 V_6(X^\prime)$ sends $A(X)$ to~$A(X^\prime)$ (resp.\ an isomorphism $f\colon V_6(X) \isomarrow V_6(X^\prime)^\vee$ with $(\wedge^3 f)\left(A(X)\right) = A(X^\prime)^\perp$). We prove that the Chow motives in middle degree of such generalised partners or duals are isomorphic. 

\begin{theorem*}
Let $X$ and~$X^\prime$ be GM varieties of dimensions $n$ and~$n^\prime$ over a field $K=\Kbar$ of characteristic~$0$ which are generalised partners or generalised duals. Then there is an isomorphism of rational Chow motives
\[
\frh^n(X) \cong \frh^{n^\prime}(X^\prime)\left(\tfrac{n^\prime-n}{2}\right).
\]
\end{theorem*}

This theorem, too, is used in an essential way in~\cite{FuMoonen-TCGMV}. The proof relies on a result of Kuznetsov and Perry, see~\cite{KuznetsovPerry-CatCone}, which says that in this situation the Kuznetsov components of~$X$ and~$X^\prime$ are equivalent.

\subsection{}
We are certainly not the first to study algebraic cycles on Gushel--Mukai varieties, and part of our work here is a refinement or completion of work done by other people. In particular, let us note that during the preparation of this paper (together with~\cite{FuMoonen-TCGMV}), a preprint~\cite{BolognesiLaterveer-GM6} by Bolognesi and Laterveer was posted that has some overlap with part of the work presented here.

It is a pleasure for us to dedicate this paper to Claire Voisin, who has contributed so much to complex geometry, Hodge theory, and the study of algebraic cycles.

\subsection{Notation.}
\label{subsec:NotConv}
For a variety~$X$, we write $\CH^i(X)$ for the Chow group in codimension~$i$ with integral coefficients, and $\CH^i(X)_\QQ = \CH^i(X) \otimes \QQ$. If $X$ is of pure dimension $n$, then $\CH_i(X):=\CH^{n-i}(X)$. We denote by $\CH(X)_{\alg} \subset \CH(X)_{\hom} \subset \CH(X)$ the subgroups of classes that are algebraically (resp.\ homologically) trivial.

If $X$ is a complete non-singular complex algebraic variety and $i$ is a natural number, we denote by $J^{2i+1}(X)$ the intermediate Jacobian in degree~$2i+1$.

\subsection*{Acknowledgment.} We thank the referee for their pertinent comments and careful reading. 

%%%%
%%%%
\section{Generalities on Gushel--Mukai varieties}

Throughout this section, $K$ denotes an algebraically closed field of characteristic~$0$.

\subsection{}
\label{subsec:GMIntro}
If $V_5$ is a $5$-dimensional $K$-vector space, let $\Grass(2,V_5) \subset \PP(\wedge^2 V_5)$ be the Grassmannian variety of $2$-planes in~$V_5$, in its Pl\"ucker embedding.

A \emph{Gushel--Mukai} (\emph{GM}) \emph{variety} of dimension $n \in \{3,4,5,6\}$ over~$K$ is a non-singular projective variety~$X$ with $\dim(X)=n$ that can be realised as a scheme-theoretic intersection
\[
X = \CGr(2,V_5) \cap \PP(W) \cap Q,
\]
where $V_5$ is a $5$-dimensional $K$-vector space, $\CGr(2,V_5) \subset \PP(K \oplus \wedge^2 V_5)$ is the cone over the Grassmannian, $W \subset K \oplus \wedge^2 V_5$ is a linear subspace of dimension~$n+5$, and $Q \subset \PP(K \oplus \wedge^2 V_5)$ is a quadric. We refer to the series of papers by Debarre and Kuznetsov (notably \cite{DK-GMClassification,DK-Kyoto,DK-GMJacobian}) for an in-depth study of such varieties and for references to the contributions by many other people. A good starting point is the overview paper~\cite{Debarre-GM}. We will follow the notation of these papers; here we only record some basic facts that we need later.

Because $X$ is non-singular, it does not contain the vertex~$O$ of $\CGr(2,V_5)$, and we have a morphism $\gamma \colon X \to \Grass(2,V_5)$, which is called the Gushel map. The corresponding rank~$2$ subbundle of $V_5 \otimes \scrO_X$ is denoted by~$\scrU_X$ and is called the Gushel bundle on~$X$.

GM varieties come in two flavours: with notation as above, we say that $X$ is
\begin{itemize}
	
\item of \emph{Mukai type} if $\gamma$ is a closed embedding, which happens if and only $O \notin \PP(W)$ (in this case, projection from~$O$ gives a realisation of~$X$ as an intersection of the Grassmannian $\Grass(2,V_5)\subset \PP(\wedge^2V_5)$ with a linear subspace and a quadric); 
	
\item of \emph{Gushel type} if $O \in \PP(W)$, in which case $\gamma$ is a double cover of its image, which is an $n$-dimensional linear section of $\Grass(2,V_5)$, denoted by~$M$. The branch divisor of this double cover $X \to M$ is an $(n-1)$-dimensional GM variety of Mukai type contained in~$M$.
\end{itemize}
In other papers these two types are referred to as ``ordinary'' and ``special'', respectively. As our work in the paper~\cite{FuMoonen-TCGMV} concerns GM varieties in characteristic~$p>0$ and the term ``ordinary'' has a well-established (different) meaning in algebraic geometry in characteristic~$p>0$, we prefer to use the above terminology.

Note that GM varieties of Gushel type with $\dim(X) < 6$ are specialisations of GM varieties of Mukai type and that all GM sixfolds are of Gushel type (as in this case necessarily $W = K \oplus \wedge^2 V_5$).

\subsection{}
It is shown in~\cite{DK-GMClassification} that to a GM variety~$X$, one can canonically associate a so-called GM data set $(V_6,V_5,W,L,q,\mu,\epsilon)$ and a corresponding Lagrangian data set $(V_6,V_5,A)$. We refer to \cite[Sections~2.2 and~3.2]{DK-GMClassification} (especially Theorems~2.9 and~3.10) for details. We only recall that $V_6$ is a $6$-dimensional $K$-vector space, $V_5 \subset V_6$ is a $5$-dimensional subspace, and $A$ is a Lagrangian subspace of~$\wedge^3 V_6$ for the natural $\det(V_6)$-valued symplectic form on~$\wedge^3 V_6$. 

The vector spaces $V_6$, $V_5$, $A$, and~$W$ can be constructed from~$X$ in a functorial way. If the context requires it, we write $V_6(X)$, $V_5(X)$, etc.

\subsection{}
Being Fano varieties, GM varieties are rationally connected, by the well-known results of Campana, see~\cite{Campana}, and Koll\'ar--Miyaoka--Mori, see~\cite{KMM}. Further, it is known that for every GM variety~$X/K$ (non-singular, as always, and in characteristic~$0$ by convention), we have $\Pic_{X/K} \cong \ZZ$; \textit{cf.} \cite[Lemma~2.29]{DK-GMClassification}. (This last result is also true in characteristic~$p>0$; see~\cite{FuMoonen-TCGMV}.)

\section{Generalities on Chow motives}

We recall an important result due to Bloch and Srinivas, see~\cite{BlochSrinivas}, which is the basis of many results about Chow motives.

\begin{theorem}[Bloch--Srinivas]
\label{thm:Bloch-Srinivas}
Let $X$ be a complete non-singular complex algebraic variety. Assume that the Chow group of\, $0$-cycles $\CH_0(X)$ is supported on an $r$-dimensional closed algebraic subset of\,~$X$.
\begin{enumerate}
\item\label{BS-1} If $r\leq 3$, then the Hodge conjecture for $H^4(X,\QQ)$ is true.
		
\item\label{BS-2} If $r\leq 2$, then the Griffiths group $\Griff^2(X)$ is zero,  and also $H^{3,0}(X)=0$; hence, the intermediate Jacobian~$J^3(X)$ is an abelian variety. 
		
\item\label{BS-3} If $r\leq 1$, then the group $\CH^2(X)_{\alg}$ of algebraically trivial cycles of codimension~$2$ is weakly representable by an abelian subvariety $J^3_{\mathrm{a}}(X)$ of the intermediate Jacobian~$J^3(X)$.
\end{enumerate}
\end{theorem}

Note that the assertions on $H^{3,0}(X)$ and $J^3(X)$ in~\ref{BS-2} are not in \cite{BlochSrinivas} but can be easily deduced; this is known as the generalised Mumford theorem; see for example \cite[Theorem~3.13]{Voisin-BookDecompDiagonal}.
Assertion~\ref{BS-3} is not explicit in~\cite{BlochSrinivas}; see however \cite[Theorem~C]{MurreKTheory}. The weak representability means, in particular, that we have an isomorphism  of groups $\CH^2(X)_{\alg} \isomarrow J^3_{\mathrm{a}}(X)(\CC)$.

\subsection{}
For a field~$k$, we denote by $\CHM(k)$ the category of Chow motives over~$k$ (with rational coefficients). We use the contravariant (cohomological) notion of motives, so that the functor sending a smooth projective variety~$X$ over~$k$ to its Chow motive~$\frh(X)$ is contravariant.

Let $\unitmot(1)$ denote the Tate motive and, as usual, write $\unitmot(m) = \unitmot(1)^{\otimes m}$. (So $\unitmot(-1)$ is the Lefschetz motive.) The Chow groups (with rational coefficients) of a motive~$M$ over~$k$ are defined by the rule $\CH^i(M)_\QQ = \Hom_{\CHM(k)}\left(\unitmot(-i),M\right)$, and $\CH(M)_\QQ = \bigoplus_{i\in \ZZ}\; \CH^i(M)_\QQ$. With this notation we have the following very useful lemma.

\begin{lemma}
\label{lem:ManinPrinc}
Let $\Omega$ be an algebraically closed field which is of infinite transcendence degree over its prime field. Let $f \colon M \to N$ be a morphism in~$\CHM(\Omega)$ such that the induced homomorphism $\CH(M)_\QQ \to \CH(N)_\QQ$ is an isomorphism. Then $f$ is an isomorphism.
\end{lemma}

\begin{proof}
This follows from \cite[Lemma~1.1]{Huybrechts-MotiveK3} together with \cite[Lemma~3.2]{Vial-Remarks}.
\end{proof}

\subsection{}
\label{subsec:CKDec}
Let $X/k$ be a smooth projective $k$-scheme of dimension~$d$. A decomposition
\[
[\Delta_X] = \sum_{i=0}^{2d}\; \pi_X^i
\]
in $\CH^d(X\times X)_\QQ$ is said to be a \emph{Chow--K\"unneth decomposition} if the~$\pi_X^i$ (viewed as correspondences from~$X$ to itself) are mutually orthogonal projectors and (for a Weil cohomology theory~$H$) the endomorphism of~$H(X)$ induced by~$\pi_X^i$ is the projection onto the summand~$H^i(X)$. Such a decomposition can also be viewed as a direct sum decomposition
\[
\frh(X) = \bigoplus_{i=0}^{2d}\; \frh^i(X)
\]
(with $\frh^i(X) = (X,\pi_X^i,0)$) such that $H\left(\frh^i(X)\right) = H^i(X)$. We shall use both points of view interchangeably. Note that the projector~$\pi_X^i$ can be viewed both as a morphism $\frh^i(X) \to \frh(X)$ and as a morphism $\frh(X) \to \frh^i(X)$. A Chow--K\"unneth decomposition as above is said to be \emph{self-dual} if ${}^{\mathsf{t}}\pi_X^i = \pi_X^{2d-i}$ for all~$i$.

The diagonal morphism $\Delta_X \colon X \to X\times X$ gives a morphism $\Delta_X^* \colon \frh(X) \otimes \frh(X) \to \frh(X)$. A Chow--K\"unneth decomposition is said to be \emph{multiplicative} if for all indices~$i$ and~$j$ the composition
\[
\frh^i(X) \otimes \frh^j(X) \tto \frh(X) \otimes \frh(X) \xrightarrow{~\Delta_X^*~} \frh(X)
\]
factors through $\frh^{i+j}(X)$. See \cite[Lemma~1.6]{FLV-MCK} for other formulations of this property. (Multiplicativity in fact implies that the decomposition is self-dual; see \cite[Proposition~1.7]{FLV-MCK}.)

\subsection{}
The subcategory of abelian motives $\CHM(k)^{\ab} \subset \CHM(k)$ is defined as the smallest full replete rigid tensor subcategory that is closed under taking direct summands and contains all motives of $1$-dimensional smooth proper $k$-schemes. (``Replete'' means that if a motive is isomorphic to an abelian motive, it is itself an abelian motive.)

A Chow motive~$M$ is said to be finite-dimensional (in the sense of Kimura--O'Sullivan) if there exist a decomposition $M = M^+ \oplus M^-$ and a positive integer~$n$ such that $\wedge^n M^+ = \Sym^n M^- = 0$. It is known that all abelian motives are finite-dimensional; see~\cite{Kimura} or~\cite[Section~2]{Andre-MotDimFin}.

We will use the following result; see \cite[Corollaire~3.16]{Andre-MotDimFin}.

\begin{theorem}
\label{thm:Conservative}
The natural functor $\CHM(k)^{\ab} \to \Mot_{\num}(k)$ from the category of abelian Chow motives to the category of numerical motives is conservative $($i.e.~it detects isomorphisms\,$)$.
\end{theorem}

We will also use the following result of Vial; see \cite[Theorem~4]{Vial-Projectors}. For an abelian variety~$J$, we define $\frh_1(J) = \frh^1(J)^\vee$, where $\frh(J) = \oplus_{i=0}^{2g}\; \frh^i(J)$ is the Deninger--Murre decomposition of the Chow motive of~$J$ as in \cite[Theorem~3.1]{DeningerMurre}.

\begin{theorem}[Vial]
\label{thm:Vial-MotDecomp}
Let X be a non-singular complex projective variety of dimension~$d$. Assume that the Abel--Jacobi map $\CH^i(X)_{\QQ,\hom}\to J^{2i-1}(X)\otimes \QQ$ is injective for all~$i$. Then $J^{2i-1}(X)$ is an abelian variety, and there is a Chow--K\"unneth decomposition $\frh(X) = \oplus_{i=0}^{2d}\; \frh^i(X)$ with
\[
\frh^{2i}(X) \cong H^{2i}(X,\QQ) \otimes \unitmot(-i),\qquad
\frh^{2i-1}(X) \cong \frh_1\left(J^{2i-1}(X)\right)(-i).
\]
In particular, $\frh(X)$ is an abelian motive.
\end{theorem}

\section{Integral Chow groups of GM varieties}
\label{sec:ChowGroupGM}

In this section we collect the computation of all the representable Chow groups, with \emph{integral} coefficients, of complex GM varieties. The  results are essentially new only in dimension~$6$. 
%\vspace{\baselineskip}

%\goodbreak\noindent
\subsection*{Gushel--Mukai threefolds.}
%\bigskip

%\noindent
The following result is essentially just an application of the Bloch--Srinivas theorem, Theorem~\ref{thm:Bloch-Srinivas}.

\begin{theorem}
\label{thm:ChowGroupGM3}
Let $X$ be a complex GM threefold, and let $J = J^3(X)$ be its intermediate Jacobian.
\begin{enumerate}
\item\label{thm:ChowGroupGM3-1} The cycle class map induces isomorphisms $\CH^i(X)\isomarrow \ZZ$ for $i=0$, $1$,~$3$.
		
\item\label{thm:ChowGroupGM3-2} For codimension $2$ cycles, algebraic equivalence and homological equivalence coincide, i.e.~$\Griff^2(X)=0$, and the Abel--Jacobi map $\CH^2(X)_{\alg} \to J(\CC)$ is an isomorphism.
		
\item\label{thm:ChowGroupGM3-3} We have $H^4(X,\ZZ) \cong \ZZ$, and this group is spanned by the cohomology class of a line on~$X$. We have a split short exact sequence
\[
0\tto J(\CC) \tto \CH^2(X) \xrightarrow{~\class~} \ZZ\tto 0.
\]
\end{enumerate}
\end{theorem}

\begin{proof}
\ref{thm:ChowGroupGM3-1} For $i=0$ this is clear, for $i=1$ this is true because every complex GM variety has Picard group~$\ZZ$, and for $i=3$ this follows from the fact that $X$ is rationally connected.
	
\ref{thm:ChowGroupGM3-2} As $\CH_0(X)$ is supported on a point, this follows from the Bloch--Srinivas theorem (Theorem~\ref{thm:Bloch-Srinivas}), taking into account that $\CH^2(X)_{\alg} \to J(\CC)$ is surjective because $X$ is rationally connected. (See \cite[Theorem~12.22]{Voisin-Book1}.)
	
\ref{thm:ChowGroupGM3-3} By \cite[Proposition~3.4]{DK-Kyoto}, the groups $H^i(X,\ZZ)$ are torsion-free and $H^2(X,\ZZ) = \ZZ \cdot \class(H)$, where $H \subset X$ is a hyperplane section for the given embedding $X \subset \PP(\CC \oplus \wedge^2 V_5)$. Further, $H^4(X,\QQ)$ is $1$\nobreakdash-dimensional, and $X$ contains lines. As any line $L \subset X$ has $L \cdot H = 1$, it follows that $H^4(X,\ZZ) = \ZZ \cdot \class(L)$. The last assertion is then immediate from~\ref{thm:ChowGroupGM3-2}.
\end{proof}

%\vspace{\baselineskip}

%\goodbreak
%\noindent
\subsection*{Gushel--Mukai fourfolds}
%\bigskip

\begin{proposition}
Let $X$ be a complex GM fourfold.
\begin{enumerate}
\item\label{thm:ChowGroupGM4-1} The cycle class map induces isomorphisms $\CH^i(X)\isomarrow \ZZ$ for $i=0$, $1$,~$4$.
		
\item\label{thm:ChowGroupGM4-2}  The cycle class map $\CH^2(X) \to H^4(X,\ZZ)\left(2\right)$ is injective, with image the subgroup of integral Hodge classes; hence, we obtain an isomorphism
\[
\CH^2(X)\isomlongarrow \left[H^4(X,\ZZ) \cap H^{2,2}(X,\CC)\right].
\]
\end{enumerate}
\end{proposition}

\begin{proof}
\ref{thm:ChowGroupGM4-1} This is true by the same argument as for Theorem~\ref{thm:ChowGroupGM3}\ref{thm:ChowGroupGM3-1}.
	
\ref{thm:ChowGroupGM4-2} Because $\CH_0(X)$ is supported on a point, Theorem~\ref{thm:Bloch-Srinivas} together with the fact that $H^3(X,\ZZ)=0$ gives that $\CH^2(X)_{\hom}=\CH^2(X)_{\alg}=0$. Hence the cycle class map in~\ref{thm:ChowGroupGM4-2} is injective. The last assertion follows from the integral Hodge conjecture for GM fourfolds, which is proven by Perry in~\cite{Perry-IHC}.
\end{proof}

\begin{remark}
  The Chow group $\CH_1(X)$ is not representable and is very much related to the $\CH_0$ of the double Eisenbud--Popescu--Walter (EPW)  sextic hyperk\"ahler fourfold. 
\end{remark}

%\vspace{\baselineskip}

%\goodbreak
%\noindent
\subsection*{Gushel--Mukai fivefolds.}
%\bigskip

%\noindent
The integral Chow groups of GM fivefolds were recently computed in the thesis of Zhou~\cite{ZhouLin}, following the strategy of Fu--Tian in \cite{FuTian-CubicFivefold}. Analogous results with $\QQ$-coefficients had been obtained by Laterveer; see~\cite{Laterveer-GM5}.

\begin{theorem}[Zhou]
\label{thm:ChowGroupGM5}
Let $X$ be a complex GM fivefold, and let $J = J^5(X)$ be the intermediate Jacobian.
\begin{enumerate}
\item The cycle class maps induce isomorphisms
\[
\CH^0(X)\cong \ZZ,\quad \CH^1(X)\cong \ZZ,\quad \CH^2(X)\cong \ZZ\oplus \ZZ,\quad \CH^4(X)\cong \ZZ,\quad \CH^5(X)\cong \ZZ.
\]
		
\item For codimension~$3$ cycles, algebraic equivalence and homological equivalence coincide; i.e.~$\Griff^3(X)=0$. The Abel--Jacobi map $\CH^3(X)_{\alg} \isomarrow J(\CC)$ is an isomorphism.
		
\item We have $H^6(X,\ZZ) \cong \ZZ \oplus \ZZ$, and there is a split short exact sequence
\[
0 \tto J(\CC) \tto \CH^3(X) \tto H^6(X,\ZZ) \tto 0.
\]
\end{enumerate}
\end{theorem}

%\vspace{\baselineskip}

%\goodbreak
%\noindent
\subsection*{Gushel--Mukai sixfolds.}
%\bigskip

%\noindent
We now turn to GM sixfolds, which require more work. The proof of the following result will take up the rest of this section. With rational coefficients, similar results were also obtained in~\cite{BolognesiLaterveer-GM6}.

\begin{theorem}
\label{thm:ChowGroupGM6}
{\samepage Let $X$ be a complex GM sixfold.
\begin{enumerate}
\item \label{thm:ChowGroupGM6-1} The cycle class maps induce isomorphisms
\[
\CH^0(X)\cong \ZZ,\quad \CH^1(X)\cong \ZZ, \quad \CH^6(X)\cong \ZZ.
\]
		
\item\label{thm:ChowGroupGM6-2} We have $H^{10}(X,\ZZ)\cong \ZZ$, generated by the class of a line contained in~$X$. The cycle class map $\CH^5(X)\to H^{10}(X,\ZZ)\left(5\right)$ is an isomorphism. 

\item\label{thm:ChowGroupGM6-3} We have $H^4(X,\ZZ) \cong \ZZ\oplus \ZZ$, generated by the classes $H^2$ and $c_2(\mathscr{U}_X)$, where $\mathscr{U}_X$ is the Gushel bundle on~$X$ and $H \in \CH^1(X) \cong \ZZ$ is the class of an ample generator. The cycle class map $\CH^2(X)\to H^4(X, \ZZ)\left(2\right)$ is an isomorphism. 
		
\item\label{thm:ChowGroupGM6-4} We have $H^6(X,\ZZ) \cong \ZZ^{24}$ with Hodge numbers $h^{4,2} = h^{2,4} = 1$ and $h^{3,3} = 22$. The cycle class map $\CH^3(X)\to H^6(X,\ZZ)\left(3\right)$ is injective.		
\end{enumerate}
}\end{theorem}

\begin{remarks}\leavevmode
\begin{enumerate}
\item The theorem does not give any information about the structure of the Chow group~$\CH^4(X)$, which is certainly the most interesting one. It is not representable and is closely related to the Chow group of $0$-cycles on the associated double EPW sextic.

\item As we shall show in the next section, the Hodge conjecture (with rational coefficients) is true for GM varieties. Moreover, $H^6(X,\ZZ)$ has no torsion (see \cite[Proposition~3.4]{DK-Kyoto}), so that the image of $\CH^3(X)\to H^6(X,\ZZ)\left(3\right)$ is a subgroup of finite index in the space of Hodge classes in $H^6(X,\ZZ)\left(3\right)$.
\end{enumerate}
\end{remarks}

\subsection{}
The proof of Theorem~\ref{thm:ChowGroupGM6} uses the same strategy as in \cite{FuTian-CubicFivefold} and~\cite{ZhouLin}, which is very much inspired by the work of Colliot-Th\'el\`ene--Voisin~\cite{CTVoisin} and Voisin~\cite{Voisin-H4nr}.

The proof of~\ref{thm:ChowGroupGM6-1} is the same as for Theorem~\ref{thm:ChowGroupGM3}\ref{thm:ChowGroupGM3-1}.

Part~\ref{thm:ChowGroupGM6-3} is an application of the Bloch--Srinivas theorem. Indeed, as $\CH_0(X)$ is supported on a point, parts~\ref{BS-2} and~\ref{BS-3} of Theorem~\ref{thm:Bloch-Srinivas} imply that $\CH^2(X)_{\hom} = \CH^2(X)_{\alg} = 0$ since $H^3(X, \ZZ)=0$. Hence the cycle class map $\CH^2(X)\to H^4(X, \ZZ)$ is injective. On the other hand, as is shown in \cite[Proposition~3.4]{DK-Kyoto}, $H^*(X,\ZZ)$ is torsion-free, and the homomorphism $\gamma^*\colon H^4\left(\Grass(2,V_5),\ZZ\right)\to H^4(X, \ZZ)$ induced by the Gushel map~$\gamma$ is an isomorphism. This gives~\ref{thm:ChowGroupGM6-3} because $H^2 - c_2(\mathscr{U}_X)$ and~$c_2(\mathscr{U}_X)$ are the images under~$\gamma^*$ of the Schubert classes  $\sigma_2$ and~$\sigma_{1,1}$, respectively, which generate $H^4\left(\Grass(2,V_5),\ZZ\right)$.

We next turn to the proof of part~\ref{thm:ChowGroupGM6-2}. Again by \cite[Proposition~3.4]{DK-Kyoto}, the assertions about $H^{10}(X,\ZZ)$ and the surjectivity of the cycle class map are clear because by~\cite[Theorem~4.7]{DK-Kyoto} $X$ contains lines. (See also below.) As for the injectivity of the cycle class map, we first need some geometric facts.

\begin{lemma}
\label{lemma:FourLines}
Any two points of $X$ can be connected by at most four lines.
\end{lemma}

\begin{proof}
Recall from Section~\ref{subsec:GMIntro} that the Gushel bundle~$\scrU_X$ on~$X$ is the pull-back of the tautological rank~$2$ bundle on the Grassmannian via the Gushel map $\gamma \colon X \to \Grass(2,V_5)$. The embedding $\scrU_X \subset V_5\otimes \scrO_X$  induces a natural morphism
\begin{equation}
\label{eq:rho1}
\rho_1\colon \PP_X(\mathscr{U}_X)\tto \PP(V_5),
\end{equation}
which is a flat morphism whose fibres are isomorphic to quadrics in~$\PP^4$. We refer to \cite[Section~4]{DK-GMClassification} and \cite[Section~2.3]{DK-Kyoto} for details. (Note that for a GM variety~$X$ of dimension~$6$, the locus that in \textit{loc.\ cit.}\ is called $\Sigma_1(X)$ is empty.)
	
Let $F_1(X)$ be the Fano variety of lines in~$X$. If $L \subset X$ is a line, then there are uniquely determined subspaces $U_1 \subset U_3 \subset V_5$ (depending on~$L$, of course) such that
\[
\gamma(L) = \left\{U_2 \in \Grass(2,V_5) \bigm| U_1 \subset U_2 \subset U_3\right\}.
\]
The map $L \mapsto U_1$ gives a morphism $\sigma \colon F_1(X)\to \PP(V_5)$. By \cite[Proposition~4.1]{DK-Kyoto}, $F_1(X)$ can be identified, as a scheme over~$\PP(V_5)$, with the relative Hilbert scheme of lines in the quadric fibration~$\rho_1$:
\begin{equation}
\label{eq:F1=Hilb}
F_1(X)\cong \Hilb^{\PP^1}\!\left(\PP_X(\mathscr{U}_X)/\PP(V_5)\right).
\end{equation}
	
Let $x$ and $x^\prime$ be points of~$X$. Write $\gamma(x)=[U_2]$ and $\gamma(x^\prime)=[U_2^\prime]$ in $\Grass(2,V_5)$. It is clear that there exists a third $2$-dimensional subspace $U_2^\pprime \subset V_5$ such that $U_2\cap U_2^\pprime$ and $U_2^\prime \cap U_2^\pprime$ are both $1$-dimensional. Let $x^\pprime \in X$ be a point in $\gamma^{-1}\left\{[U_2^\pprime]\right\}$. Let $y$ be a generator of $U_2\cap U_2^\pprime$, which defines a point $[y] \in \PP(V_5)$.
	
By construction, $(x, [y])$ and $(x^\pprime, [y])$ are points of $\PP_X(\mathscr{U}_X)$ that lie on the fibre of~$\rho_1$ over~$[y]$. This fibre is a quadric of dimension~$3$; hence any two points on it can be connected by at most two lines in it. Therefore, $x$ and $x^\pprime$ can be connected by two lines in~$X$. Similarly, $x^\prime$ and $x^\pprime$ can be connected by at most two lines, and we are done.
\end{proof}

\begin{remark}
Lemma~\ref{lemma:FourLines} will be applied in Proposition~\ref{prop:CH1GM6BoundedTorsion}, where only the following weaker assertion is needed: \textit{any two points of\, $X$ can be connected by finitely many lines.} Let us sketch an alternative argument for this fact that was kindly indicated to us by the referee. First, it is easy to see that through every point of $X$ passes a line. Then, by \cite[Theorem~IV.4.16]{KollarBookRationCurve}, one can perform the rationally connected quotient $X\dashrightarrow Z$ for the (pre)relation of line-chain connectedness in $X$. We claim that $Z$ is a point. Suppose $\dim(Z)\geq 1$, and consider the Zariski closure $D$ of the inverse image of a prime divisor of $Z$. Since $\operatorname{N}_1(X)\cong \mathbb{R}$, the divisor $D$ is ample; on the other hand, $(D\cdot L)=0$ for  any line $L$ in $X$. This gives a contradiction. This argument is also in \cite[Corollary~IV.4.14]{KollarBookRationCurve}.
\end{remark}

\begin{lemma}
\label{lem:F1IrredRC}
The variety $F_1(X)$ of lines on~$X$ is rationally chain connected. In particular, all lines contained in $X$ have the same class in $\CH_1(X)$.
\end{lemma}

Note that we do not claim that $F_1(X)$ is irreducible; this is probably true, but we do not need it.

\begin{proof}
We use~\eqref{eq:F1=Hilb}. By \cite[Proposition~4.5]{DK-GMClassification}, the fibres of the map~\eqref{eq:rho1} are quadrics in~$\PP^4$ of corank at most~$3$. If $Q \subset \PP^4$ is a non-degenerate quadric, its Hilbert scheme of lines is isomorphic to~$\PP^3$. If $Q$ has corank~$1$, its Hilbert scheme is the union of two $\PP^1$-bundles over~$\PP^1$, glued along a $2$-dimensional quadric. If the corank is~$2$, the Hilbert scheme is obtained from a $\PP^2$-bundle over~$\PP^1$ by contracting a section, and if the corank is~$3$, we get a union of two copies of~$\PP^3$, glued along a~$\PP^2$. So in all cases the fibres are rationally chain connected. Moreover, there is a non-empty open subset $U \subset \PP(V_5)$ over which $\sigma$ is a $\PP^3$-bundle. Let $Z \subset F_1(X)$ be the closure of~$\sigma^{-1}(U)$. By \cite[Corollary~1.3]{GraberHarrisStarr}, $Z$ is rationally connected, and since $Z$ meets all fibres, the assertion follows.
\end{proof}

\begin{proposition}
\label{prop:CH1GM6BoundedTorsion}
There exists an integer $N>0$ such that for any $\alpha\in \CH_1(X)$ the class~$N\alpha$ is a multiple of the class of a line in~$X$. In particular, $\CH_1(X)_\QQ\cong \QQ$, generated by the class of a line.
\end{proposition}

\begin{proof}
The following argument of ``successive ruled surfaces'' is taken from \cite[Proposition~3.1]{TianZong}, which is essentially \cite[Proposition IV.3.13.3]{KollarBookRationCurve}. We reproduce it here for the reader's convenience. Fix a point $x_0 \in X$. Let $I = \left\{(L,x) \in F_1(X) \times X \mid x\in L\right\}$ be the incidence variety, and let $I^{(2)} = I \times_{F_1(X)} I = \left\{(L,x_1,x_2) \in F_1(X) \times X \times X \mid x_1, x_2 \in L\right\}$. Let $e \colon I \to X$ and $e_1$, $e_2 \colon I^{(2)} \to X$ be the evaluation morphisms. Then
\[
B = e_1^{-1}\{x_0\}\; {}_{e_2}\!\times_{e_1} I^{(2)} {}_{e_2}\!\times_{e_1} I^{(2)} {}_{e_2}\!\times_{e} I
\]
is the scheme of tuples $(L_1,x_1,L_2,x_2,L_3,x_3,L_4)$ with $x_0 \in L_1$ and $x_j \in L_j\cap L_{j+1}$ for $j=1,2,3$. For $j=1,\ldots,4$, let
\[
\scrL_j = \left\{\left(\left(L_1,x_1,L_2,x_2,L_3,x_3,L_4\right),y\right) \in B \times X \bigm| y \in L_j \right\}.
\]
The first projection makes $\scrL_j$ a $\PP^1$-bundle over~$B$. The second projection gives an evaluation morphism $e_j \colon \scrL_j \to X$. Furthermore, we have sections
\[
\begin{tikzcd}
\scrL_1 \ar[d] & \scrL_2 \ar[d] & \scrL_3 \ar[d] & \scrL_4 \ar[d]\\
B \ar[u,bend left,"s_0"] \ar[u,bend right,"s_1^\prime"'] &
B \ar[u,bend left,"s_1"] \ar[u,bend right,"s_2^\prime"'] &
B \ar[u,bend left,"s_2"] \ar[u,bend right,"s_3^\prime"'] &
B \ar[u,bend left,"s_3"]
\end{tikzcd}
\]
where $s_j$ ($j=0,1,2,3$) is obtained by taking $y=x_j$, viewed as a point of~$L_{j+1}$, and $s_j^\prime$ ($j=1,2,3$) is obtained by taking $y=x_j$, viewed as a point of~$L_j$. Let $\pi \colon \scrL \to B$ be the scheme over~$B$ obtained by gluing $\scrL_j$ and~$\scrL_{j+1}$ along their sections~$s_j^\prime$ and~$s_j$. The fibres of~$\pi$ are chains of four lines connected at points. By construction, the morphisms~$e_j$ glue to an evaluation morphism $e \colon \scrL \to X$, which by Lemma~\ref{lemma:FourLines} is surjective. By taking  general successive hyperplane sections of $\scrL$, we obtain a generically finite morphism $\scrL'\to X$, whose degree is denoted by $m\in \ZZ_{>0}$. Set $N=m!$.
	
Without loss of generality, we may assume that $\alpha$ is the class of an irreducible curve $C \subset X$. Our goal is to show that $N[C]\in \CH_1(X)$ is a linear combination of classes of lines. We may assume $C$ is not itself a line. There exists an irreducible curve $\hat{C} \subset \scrL'$ such that $e$ restricts to a non-constant morphism $\hat{C} \to C$. Because $C$ is not a line, $\pi(\hat{C})$ is a curve. Let $B_0$ be the normalisation of this curve, and define $Y = B_0 \times_B \scrL$, which is a union of four ruled surfaces~$Y_j$ over~$B_0$, glued along sections. The normalisation of~$\hat{C}$ maps to~$Y$; let $\tilde{C} \subset Y$ be the image. By construction, the evaluation map $e \colon Y \to X$ gives a non-constant morphism $\tilde{C} \to C$, so that $e_*[\tilde{C}] = m'\cdot [C] \in \CH_1(X)$ for some $0<m'\leq m$.

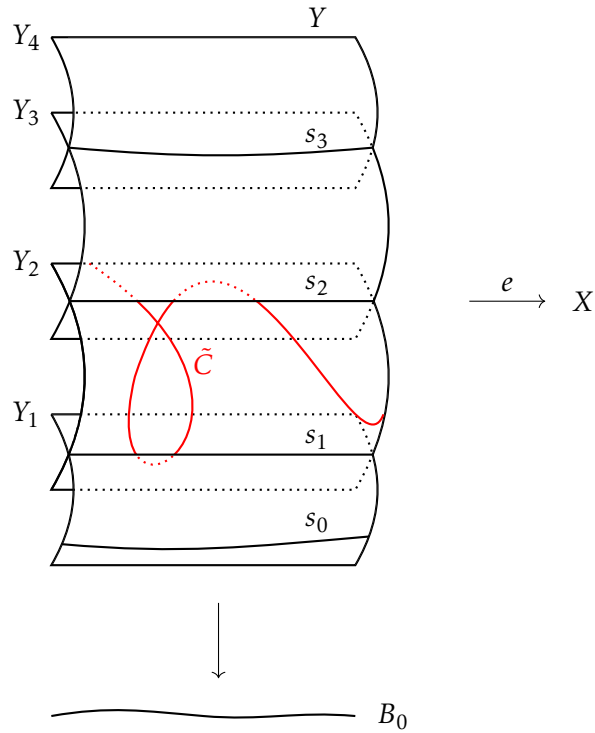
\[
\begin{tikzpicture}
\coordinate (LA) at (0,8) {};
\coordinate (LB) at (0,7) {};
\coordinate (LC) at (0,6) {};
\coordinate (LD) at (0,5) {};
\coordinate (LE) at (0,4) {};
\coordinate (LF) at (0,3) {};
\coordinate (LG) at (0,2) {};
\coordinate (LH) at (0,1) {};
\coordinate (RA) at (4,8) {};
\coordinate (RB) at (4,7) {};
\coordinate (RC) at (4,6) {};
\coordinate (RD) at (4,5) {};
\coordinate (RE) at (4,4) {};
\coordinate (RF) at (4,3) {};
\coordinate (RG) at (4,2) {};
\coordinate (RH) at (4,1) {};
	
\path[save path=\compd] (LA) to[in=60,out=-60] (LC) to (RC) to[in=-60,out=60] (RA) to (LA);
\path[save path=\compc] (LB) to[in=60,out=-60] (LE) to (RE) to[in=-60,out=60] (RB) to (LB);
\path[save path=\compb] (LD) to[in=60,out=-60] (LG) to (RG) to[in=-60,out=60] (RD) to (LD);
\path[save path=\compa] (LF) to[in=60,out=-60] (LH) to (RH) to[in=-60,out=60] (RF) to (LF);
\fill[white][use path=\compa];
\fill[white][use path=\compb];
\fill[white][use path=\compc];
\fill[white][use path=\compd];
\draw[thick][use path=\compa];
\draw[thick][use path=\compb];
\draw[thick][use path=\compc];
\draw[thick][use path=\compd];

\path[name path=LAC] (LA) to[in=60,out=-60] (LC);
\path[name path=LBE] (LB) to[in=60,out=-60] (LE);
\path[name path=LDG] (LD) to[in=60,out=-60] (LG);
\path[name path=LFH] (LF) to[in=60,out=-60] (LH);
\path [name intersections={of=LAC and LBE}];
\coordinate (IBC) at (intersection-1);
\path [name intersections={of=LBE and LDG}];
\coordinate (IDE) at (intersection-1);
\path [name intersections={of=LDG and LFH}];
\coordinate (IFG) at (intersection-1);
	
\path[name path=RAC] (RA) to[in=60,out=-60] (RC);
\path[name path=RBE] (RB) to[in=60,out=-60] (RE);
\path[name path=RDG] (RD) to[in=60,out=-60] (RG);
\path[name path=RFH] (RF) to[in=60,out=-60] (RH);
\path [name intersections={of=RAC and RBE}];
\coordinate (JBC) at (intersection-1);
\path [name intersections={of=RBE and RDG}];
\coordinate (JDE) at (intersection-1);
\path [name intersections={of=RDG and RFH}];
\coordinate (JFG) at (intersection-1);

\path[name path=XX] (-1,1.25) to (5,1.4);
\path [name intersections={of=LFH and XX}];
\coordinate (LS) at (intersection-1);
\path [name intersections={of=RFH and XX}];
\coordinate (RS) at (intersection-1);
	
%% path of the curve hat{C}
\path[save path=\Chat] (.5,5) .. controls (3.5,3) and (.5,1) .. (1.1,3.5) .. controls (2.25,7) and (4,2) .. (4.37,3);
	
\begin{scope}
\clip[use path=\compd];
\clip (IBC) rectangle (5,8);
\fill[white][use path=\compd];
\draw[thick,dotted][use path=\compc];
\draw[thick][use path=\compd];
\end{scope}
	
\begin{scope}
\clip[use path=\compc];
\clip (IDE) rectangle (JBC);
\fill[white][use path=\compc];
\draw[thick,dotted][use path=\compd];
\draw[thick,dotted][use path=\compb];
\draw[thick][use path=\compc];
\draw[thick,red,dotted][use path=\Chat];
\end{scope}
	
\begin{scope}
\clip[use path=\compa];
\fill[white][use path=\compa];
\draw[thick,dotted][use path=\compb];
\draw[thick][use path=\compa];
\draw[thick,red,dotted][use path=\Chat];
\end{scope}
	
\begin{scope}
\clip[use path=\compb];
\clip (IFG) rectangle ($(JDE)+(1,0)$);
\fill[white][use path=\compb];
\draw[thick,dotted][use path=\compc];
\draw[thick,dotted][use path=\compa];
\draw[thick][use path=\compb];
\draw[thick,red][use path=\Chat];
\end{scope}
	
\draw[thick] (LD) to[in=60,out=-60] (LG);
	
%% draw the sections s1,s2,s3
\draw[thick] (IBC) to[out=-5,in=185] (JBC);
\draw[thick] (IDE) to (JDE);
\draw[thick] (IFG) to (JFG);
%% draw the section s0
\draw[thick] (LS) to[out=-5,in=185] (RS);
	
%% put labels to indicate the components
\node[left] at (LA) {$Y_4$};
\node[left] at (LB) {$Y_3$};
\node[left] at (LD) {$Y_2$};
\node[left] at (LF) {$Y_1$};
	
%% put labels above the sections
\node[above] at (3.5,6.4) {$s_3$};
\node[above] at (3.5,4.45) {$s_2$};
\node[above] at (3.5,2.4) {$s_1$};
\node[above] at (3.5,1.3) {$s_0$};
	
%% put label at curve Chat
\node[above,red] at (2,3.4) {$\tilde{C}$};
	
%% put labels X and Y and draw arrow
\node[above] at (3.5,8) {$Y$};
\draw[->] (5.5,4.5) -- (6.5,4.5);
\node[above] at (6,4.5) {$e$};
\node at (7,4.5) {$X$};
	
%% draw the base curve and the map to it
\draw[->] (2.2,.5) -- (2.2,-.5);
\node at (4.5,-1) {$B_0$};
\draw[thick] (0,-1) to[out=10,in=175] (2,-1) to[out=-5,in=175] (4,-1);
\end{tikzpicture}
\]
	
In each component~$Y_j$ ($j=1,\ldots,4$), every ruling pushes forward to a line in~$X$. If $t$ is any section of~$Y_j$, then $\CH_1(Y_j)$ is generated by the class of $t(B_0)$ together with the rulings. Because $s_0$ comes from the constant section~$x_0$, we have $e_*\left[s_0(B_0)\right] = 0$. Therefore, $e_*\left[s_j(B_0)\right] \in \CH_1(X)$ lies in the subgroup spanned by the classes of lines, for each $j=1,\ldots,4$. As $\tilde{C}$ is contained in one of the components~$Y_j$, it follows that $m'\cdot [C] = e_*[\tilde{C}]$, hence $N\cdot [C]$, lies in this subgroup, which is what we wanted to prove.
\end{proof}

Now we are ready to compute the Chow group of 1-cycles of a GM sixfold.

\begin{proof}[Proof of Theorem~\ref{thm:ChowGroupGM6}\,\ref{thm:ChowGroupGM6-2}]
By Proposition~\ref{prop:CH1GM6BoundedTorsion}, there is an integer $N>0$ such that multiplication by $N$ kills $\CH_1(X)_{\alg}$. Combining this with the divisibility of $\CH_1(X)_{\alg}$, we conclude that $\CH_1(X)_{\alg} = 0$.
	
On the other hand, by the blow-up formula, the Griffiths group of $1$-cycles is a birational invariant for smooth projective varieties. Because $X$ is rational, it follows that $\Griff_1(X)=\Griff_1(\PP^6)=0$; hence, 
\[
\CH_1(X)_{\hom} = \CH_1(X)_{\alg} = 0.
\]
In other words, the cycle class map $\CH_1(X)\to H^{10}(X,\ZZ)$ is injective. The surjectivity follows from the fact that $H^{10}(X,\ZZ)$ is generated by the class of a line in $X$.
\end{proof}

Finally, we turn to part~\ref{thm:ChowGroupGM6-4} of Theorem~\ref{thm:ChowGroupGM6}. 

\begin{lemma}
\label{lem:DODGM6}
There exist an integer $n > 0$, a $($possibly reducible$)$ smooth projective variety~$T$ of dimension~$4$, a generically injective morphism $j\colon T\to X$,  and an algebraic cycle $Z\in \CH^4(X\times T)$, such that we have the following equality in $\CH^6(X\times X)$:
\begin{equation}
\label{eq:DODGM6}
n\cdot [\Delta_X] = n\cdot \left([\pt] \times [X]\right) + n\cdot \left([\operatorname{line}]\times H\right) + (\id_X \times j)_*(Z).
\end{equation}
\end{lemma}

\begin{proof}
Since we have already proven that $\CH_0(X)\cong \ZZ$ and $\CH_1(X)\cong \ZZ$, Laterveer's refined decomposition of the diagonal \cite[Theorem 1.7]{Laterveer-SmallChow} applies. This gives that there exists an integer $n > 0$ such that in $\CH^6(X\times X)$ we have a relation
\[
n\cdot [\Delta_X] = n\cdot \left([\pt]\times [X]\right) + n\cdot \left([\operatorname{line}]\times H\right) + Z^\prime,
\]
where $Z^\prime$ is supported on $X \times T^\prime$ for some closed algebraic subset $T^\prime \subset X$ of codimension~$2$. Now let $T$ be the disjoint union of resolutions of the irreducible components of~$T^\prime$, and let~$Z$ be an algebraic cycle in $X\times T$ which pushes forward to~$Z^\prime$.
\end{proof}

\begin{proposition}
\label{prop:Griff3GM6}
Algebraic and homological equivalence coincide for cycles of codimension~$3$ on a GM sixfold~$X$; i.e.~$\Griff^3(X)=0$.
\end{proposition}

\begin{proof}
We first show that $\Griff^3(X)$ is torsion. Let both sides of~\eqref{eq:DODGM6}, viewed as correspondences from~$X$ to itself, act on~$\Griff^3(X)$. For any $\alpha\in \Griff^3(X)$, the correspondence $n\cdot [\Delta_X]$ sends~$\alpha$ to $n\cdot \alpha$. The first two terms of the right-hand side of~\eqref{eq:DODGM6} send $\alpha$ to zero for dimension reasons, and the third term sends~$\alpha$ to $j_*\left(Z_*(\alpha)\right)$. Therefore, we have
\[
n\cdot \alpha = j_*(Z_*(\alpha))
\]
in $\Griff^3(X)$. However, $Z_*(\alpha)$ is an element of~$\Griff^1(T)$, which is trivial as homological equivalence and algebraic equivalence coincide for divisors. We conclude that $\Griff^3(X)$ is killed by~$n$.
	
It remains to show that $\Griff^3(X)$ is torsion-free. For any abelian group~$A$, let $\scrH^i_A$ be the Zariski sheaf associated to the presheaf $U \mapsto H^i(U,A)$. Bloch and Ogus, see~\cite{Bloch-Ogus}, showed that, starting from the $E_2$-page, the coniveau spectral sequence agrees with the Leray spectral sequence associated to the continuous map $X(\CC)\to X_{\operatorname{Zar}}$ and the constant sheaf~$A$. Therefore, we have 
\begin{equation}
\label{eq:SpectralSeq}
E_2^{p,q} = H^p\left(X,\scrH^q_A\right) \Longrightarrow N^pH^{p+q}(X, A),
\end{equation}
where $N^\bullet$ denotes the coniveau filtration.
	
We need two basic properties of this spectral sequence. 
\begin{itemize}
\item In \eqref{eq:SpectralSeq}, $E_2^{0, q}=H^0(X, \scrH^q_{A})$ is the so-called unramified cohomology $H^i_{\nr}(X,A)$, which is a birational invariant; see \cite[Theorem~2.8]{CTVoisin}. As $X$ is rational (see \cite[Proposition~4.2]{DK-GMClassification}), its unramified cohomology groups all vanish except in degree $0$; \textit{i.e.}~$E_2^{0,q} = 0$ for any $q>0$.
		
\item For $p>q$, we have $E_2^{p,q} = 0$. This is a consequence of the Gersten conjecture for homology theory proved by Bloch--Ogus; \textit{cf.}~\cite[Equation~(0.3)]{Bloch-Ogus}.
\end{itemize}
	
Using these properties, if we take $A=\ZZ$, \eqref{eq:SpectralSeq} gives rise to an exact sequence
\begin{equation}
\label{eq:LESfromSS}
0\tto H^5(X,\ZZ)/N^2H^5(X,\ZZ) \tto H^1\left(X,\scrH^4_\ZZ\right) \tto H^3\left(X,\scrH^3_\ZZ\right) \tto H^6(X, \ZZ),
\end{equation}
where $H^3(X,\scrH^3_\ZZ)$ is identified with the group of codimension~$3$ cycles modulo algebraic equivalence (see~\cite[Section~(0.5)]{Bloch-Ogus}) and the last arrow is the cycle class map. Therefore, the kernel of the last arrow is exactly the Griffiths group~$\Griff^3(X)$.
	
Since $H^5(X,\ZZ)=0$, the first term in \eqref{eq:LESfromSS} vanishes, and we obtain that 
\begin{equation}
\label{eq:Griff3=H1H4}
H^1\left(X,\scrH^4_\ZZ\right)\cong \Griff^3(X).
\end{equation}
Now we use a result of Colliot-Th\'el\`ene and Voisin \cite[Theorem 3.1]{CTVoisin}, which is based on the Bloch--Kato conjecture (proven by Voevodsky, see \cite{Voevodsky-IHES,Voevodsky-Annals}). This result implies that for any integer~$n$, we have a short exact sequence of Zariski sheaves
\[
0 \tto \scrH^4_\ZZ \xrightarrow{\;\cdot n\;} \scrH^4_\ZZ \tto \scrH^4_{\ZZ/n\ZZ} \tto 0.
\]
From the associated long exact sequence, we obtain a short exact sequence
\[
0 \tto H^0\left(X,\scrH^4_\ZZ\right)/n \tto H^0\left(X,\scrH^4_{\ZZ/n\ZZ}\right) \tto H^1\left(X,\scrH^4_\ZZ\right)[n] \tto 0,
\]
where the last term denotes the $n$-torsion subgroup of $H^1(X,\scrH^4_\ZZ)$. Now observe that the middle term is the unramified cohomology group $H^4_{\nr}(X,\ZZ/n\ZZ)$, which is trivial since $X$ is rational. It follows that $H^1(X, \scrH^4_\ZZ)$ has no $n$-torsion. On the other hand, we had already shown that $\Griff^3(X)$ is killed by~$n$; so $\Griff^3(X) = 0$.
\end{proof}

\begin{proof}[Proof of Theorem~\ref{thm:ChowGroupGM6}\,\ref{thm:ChowGroupGM6-4}]
Let both sides of \eqref{eq:DODGM6} act on $\CH^3(X)_{\alg}$. The same argument as in the first part of the proof of Proposition~\ref{prop:Griff3GM6} shows that the multiplication by~$n$ map on $\CH^3(X)_{\alg}$ factors through 
\[
Z_*\colon \CH^3(X)_{\alg}\tto \CH^1(T)_{\alg}.
\]
However, we have the commutative diagram
\begin{equation*}
\begin{tikzcd}
\CH^3(X)_{\alg}  \ar[r , "Z_*"]  \ar[d]& \CH^1(T)_{\alg} \ar[d, "\wr"]\\
J^5(X) \ar[r, "\class(Z)_*"]& \Pic^0(T)\rlap{,}
\end{tikzcd}
\end{equation*}
where the vertical arrows are Abel--Jacobi maps. Since the right vertical arrow is an isomorphism and $J^5(X)=0$ (since $H^5(X,\ZZ)=0$), the top arrow $Z_*\colon \CH^3(X)_{\alg}\to \CH^1(T)_{\alg}$ is zero. Hence $\CH^3(X)_{\alg}$ is killed by~$n$. On the other hand, $\CH^3(X)_{\alg}$ is a divisible group; hence $\CH^3(X)_{\alg} = 0$. Combining this with Proposition~\ref{prop:Griff3GM6}, we conclude that $\CH^3(X)_{\hom}=0$; \textit{i.e.}~the cycle class map is injective.
\end{proof}

\section{The generalised Hodge conjecture}
\label{sec:GHC}

We refer to \textit{the generalised Hodge conjecture} as the one proposed by Grothendieck in \cite{Grothendieck} as an amendment of Hodge's initial generalisation. The main result of this section is the following.

\begin{theorem}
\label{thm:GHC}
Let $X$ be a complex Gushel--Mukai variety. Then the generalised Hodge conjecture is true for~$X$.
\end{theorem}

\begin{proof}
  Most of this can be extracted from the literature together with the computation of the Chow groups of~$X$. For GM varieties of dimension $3$ or~$4$, the result is proven in \cite[Corollary~2.5(ii)]{Laterveer-SmallChow}. For GM varieties of dimension~$5$, the result can be found in \cite[Remark~3.2]{Laterveer-GM5} (which refines a result by Nagel in~\cite{Nagel-GHC}). For GM varieties of dimension~$6$, Theorem~\ref{thm:GHC} follows from Theorem~\ref{thm:ChowGroupGM6} (or rather Proposition~\ref{prop:CH1GM6BoundedTorsion}) together with \cite[Proposition~2.4(ii)]{Laterveer-SmallChow}.
\end{proof}

\begin{remark}
\label{rem:HCGM6alt}
Both in \cite[Remark~2.26]{KP16} and in \cite[Remark~4.2]{Debarre-GM}, it is stated that the (usual) Hodge conjecture for GM sixfolds can be proven using the results of~\cite{DK-Kyoto}, but no details are provided. While it is clear how to proceed for general~$X$ (see below), we have not been able to make this method work for all GM sixfolds.
	
First assume that $X$ is a GM variety of dimension~$6$ which is general, in the sense that condition~(10) in \cite[Section~5.1]{DK-Kyoto} is satisfied. As in~\cite{DK-Kyoto}, let $F^\sigma_2(X)$ be the Hilbert scheme of $\sigma$-planes in~$X$, and let $\scrL^\sigma_2(X)$ be the universal plane over~$F^\sigma_2(X)$. We then have a diagram
\[
\begin{tikzcd}
& \scrL_2^\sigma(X) \ar[dl,"q"']\ar[dr,"p"]\\
X && F^\sigma_2(X) \ar[d,"\tilde\sigma"]\\
&& \widetilde{Y}_{A,V_5} \ar[r,hook,"\iota"] & \widetilde{Y}_A\rlap{,}
\end{tikzcd}
\]
where $\widetilde{Y}_{A,V_5} \subset \widetilde{Y}_A$ is a subvariety of codimension~$1$ and $\tilde\sigma \colon F^\sigma_2(X) \to \widetilde{Y}_{A,V_5}$ is a $\PP^1$-bundle. By \cite[Corollary~5.13]{DK-Kyoto}, $F^\sigma_2(X)$ is a non-singular fourfold, and hence $\scrL^\sigma_2(X)$ is non-singular of dimension~$6$.
	
Let $[\mathbf{P}] \in H^6\left(F^\sigma_2(X),\QQ\right)$ denote the class of a fibre of~$\tilde\sigma$, and define $H^2\left(F^\sigma_2(X),\QQ\right)_0 = [\mathbf{P}]^\perp \subset H^2\left(F^\sigma_2(X),\QQ\right)$. By \cite[Proposition~5.14]{DK-Kyoto}, $(\iota\circ \tilde\sigma)^*$ gives an isomorphism
\[
H^2\left(\widetilde{Y}_A,\QQ\right) \isomlongarrow H^2\left(F^\sigma_2(X),\QQ\right)_0.
\]
Let $H^2(\widetilde{Y}_A,\QQ)_0 \subset H^2(\widetilde{Y}_A,\QQ)$ be the primitive subspace, and write $H^2\left(F^\sigma_2(X),\QQ\right)_{00}$ for its image under $(\iota\circ \tilde\sigma)^*$. By \cite[Theorem~5.19]{DK-Kyoto}, we then have a commutative diagram
\[
\begin{tikzcd}
H^2\left(\widetilde{Y}_A,\QQ\right) \ar[r,"\sim","~~(\iota\circ \tilde\sigma)^*~~"']& H^2\left(F^\sigma_2(X),\QQ\right)_0\\
H^2\left(\widetilde{Y}_A,\QQ\right)_0 \ar[r,"\sim"] \arrow[symbol]{u}{\bigcup} & H^2\left(F^\sigma_2(X),\QQ\right)_{00}  \arrow[symbol]{u}{\bigcup} & H^6(X,\QQ)_{00}\rlap{.} \ar[l,"\sim"]\\[-16pt]
& p_*\left(q^*(z)\right) & z \ar[l,maps to]
\end{tikzcd}
\]
To deduce the Hodge conjecture for~$X$, it now suffices to show that there exists a class $\xi \in \CH^6\left(X \times F^\sigma_2(X)\right)$ such that the inverse isomorphism $H^2\left(F^\sigma_2(X),\QQ\right)_{00} \isomarrow H^6(X,\QQ)_{00}$ is given by $y \mapsto \pr_{X,*}\left(\class(\xi) \cup \pr_{F}^*(y)\right)$, where $\pr_X \colon X \times F^\sigma_2(X) \to X$ and $\pr_F \colon X \times F^\sigma_2(X) \to F^\sigma_2(X)$ are the projections. This is proven as follows.
	
Let $h \in H^2(X,\QQ)$ be the class of the very ample line bundle $H = \scrO_X(1)$. As explained before Theorem~5.19 in \cite[Section~5]{DK-Kyoto}, $p\colon \scrL_2^\sigma(X) \to F^\sigma_2(X)$ is a $\PP^2$-bundle for which $q^*(h)$ is a relative hyperplane class. Hence there is a vector bundle~$\scrE$ of rank~$3$ on~$F^\sigma_2(X)$ such that $\scrL_2^\sigma(X) \cong \PP(\scrE)$. As $q\colon \scrL_2^\sigma(X) \to X$ is generically finite of degree~$12$ (\textit{cf.} \cite[Lemma~5.15]{DK-Kyoto}), it follows from the proof of \cite[Theorem~5.19]{DK-Kyoto} that the class
\[
\xi = \tfrac{1}{12} \cdot \left[q^*(h^2) + p^*(c_1(\scrE))\cdot q^*(h) + p^*(c_2(\scrE))\right]
\]
(viewed as a class in $\CH^6\left(X \times F^\sigma_2(X)\right)$) has the required property. This completes the argument in case $X$ is general.
	
If $X$ is arbitrary, then, by considering a family of GM varieties, we can still show that there exists a class $\gamma \in \CH_4(\widetilde{Y}_A \times X)$ that induces an isomorphism $H^2\left(\widetilde{Y}_A,\QQ(1)\right)_0 \isomarrow H_6\left(X,\QQ(-3)\right)_{00}$ given by $z\mapsto \pr_{2,*}\left(\pr_1^*(z) \cap [\gamma]\right)$. (The reader will hopefully be able to guess the meaning of the notation.) However, it is not clear to us if the Hodge conjecture in \emph{cohomological} degree~$2$ is true for the variety~$\widetilde{Y}_A$, as in general this variety is singular.
\end{remark}

\section{The Mumford--Tate conjecture and the generalised Tate conjecture}
\label{sec:MTC+GTC}

In this section we first recall the statements of the Mumford--Tate conjecture and the generalised Tate conjecture (in characteristic~$0$). The main result that we prove is that these conjectures are true for Gushel--Mukai varieties of even dimension. We deduce this from a theorem of Andr\'e; see~\cite{Andre-ShafTate}. This argument uses that, as shown in the previous section, the Hodge conjecture is true for these varieties.

In all of this section, $\ell$ is a fixed (but arbitrary) prime number.

\subsection{}
\label{subsec:MTC}
As explained in \cite[Section~1]{Moonen-TCMTC}, the Mumford--Tate conjecture and the Tate conjecture can be viewed as statements about complex algebraic varieties, and we will take that perspective.

Let  $X$ be a complete non-singular algebraic variety over~$\CC$. Fix an integer~$i$, and write~$H_\ell$ for $H^{2i}\left(X,\QQ_\ell(i)\right)$. Choose any subfield $F \subset \CC$ that is finitely generated over~$\QQ$ and a model~$X_F$ of~$X$ over~$F$. As $H_\ell$ is canonically isomorphic to $H^{2i}\left(X_{\Fbar},\QQ_\ell(i)\right)$, the choice of the model~$X_F$ gives us a continuous Galois representation
\[
\rho_\ell \colon \Gal\left(\Fbar/F\right) \tto \GL(H_\ell).
\]
Define $G_\ell$ to be the Zariski closure of the image of~$\rho_\ell$, and let $G_\ell^0$ denote its identity component. The algebraic subgroup $G_\ell^0 \subset \GL\left(H_\ell\right)$ only depends on~$X$ and is independent of the choice of~$F$ and~$X_F$; see \cite[Proposition~1.3]{Moonen-TCMTC}.

Let $H_{\Betti} = H^{2i}\left(X(\CC),\QQ(i)\right)$, and let $G_{\Betti} \subset \GL(H_{\Betti})$ be its Mumford--Tate group. (Here ``$\Betti$'' is for ``Betti realisation''.) Artin's comparison isomorphism $H_{\Betti} \otimes \QQ_\ell \isomarrow H_\ell$ induces an isomorphism of algebraic groups $\GL(H_{\Betti}) \otimes \QQ_\ell \isomarrow \GL(H_\ell)$. The \emph{Mumford--Tate conjecture} (for the chosen $X$, $\ell$, and~$i$) is the assertion that this isomorphism restricts to an isomorphism
\[
G_{\Betti} \otimes \QQ_\ell \xrightarrow[?]{\;\lowsim\;} G_\ell^0.
\]

\subsection{}
With notation as above, a class $\xi \in H_\ell$ is called a Tate class if it is fixed under the action of~$G_\ell^0$; this is equivalent to the condition that the stabiliser of~$\xi$ in $\Gal(\Fbar/F)$ is an open subgroup. All elements in the image of the cycle class map $\class_\ell \colon \CH^i(X) \otimes \QQ_\ell \to H_\ell$ are Tate classes.

The \emph{Tate conjecture} (again: for the chosen $X$, $\ell$, and~$i$) says that the algebraic group~$G_\ell^0$ is reductive and that the $\ell$-adic cycle class map
\begin{equation}
\label{eq:classl}
\class_\ell \colon \CH^i(X) \otimes \QQ_\ell \tto \{\text{Tate classes in~$H_\ell$}\}
\end{equation}
is surjective.

\subsection{}
\label{subsec:GTC}
Retaining the above notation and assumptions, let $W \subset H_\ell = H^{2i}\left(X,\QQ_\ell(i)\right)$ be a $G_\ell^0$-subrepresentation. If $r$ is a natural number, $W$ is said to be of Tate coniveau at least~$r$ if there exist a normal domain $R\subset \CC$ which is of finite type over~$\ZZ$ and a smooth proper model $X_R \to \Spec(R)$ of~$X$ over~$R$, such that
\begin{enumerate}[label=(\alph*)]
\item\label{coniv2} over the fraction field of~$R$, the group $G_\ell$ as above is connected;

\item\label{coniv3} at every closed point~$x$ of $\Spec\left(R[1/\ell]\right)$, all eigenvalues of the Frobenius at~$x$ acting on $W(r-i)$ and on~$W^\vee(r-i)$ are algebraic integers.
\end{enumerate}
(For any $X_R/R$ as above, condition~\ref{coniv2} is satisfied after replacing~$R$ by a finite extension.)

If $Z \subset X$ is a closed subscheme of which all components have codimension at least~$r$, then the kernel $\Ker\left(H_\ell \to H^{2i}((X\setminus Z),\QQ_\ell(i)) \right)$ is a $G_\ell^0$-subrepresentation of Tate coniveau at least~$r$.

The \emph{generalised Tate conjecture} (which goes back to Grothendieck, see \cite[Section~10.3]{Brauer3}) states that, conversely, every $G_\ell^0$-subrepresentation $W \subset H_\ell$ of Tate coniveau at least~$r$ is supported on a subscheme $Z \subset X$ whose components have codimension at least~$r$, in the sense that $W \subset \Ker\left(H_\ell \to H^{2i}((X\setminus Z),\QQ_\ell(i)) \right)$. (For $r=i$, this says that the $\ell$-adic cycle class map~\eqref{eq:classl} is surjective.)

\begin{remark}
The Mumford--Tate conjecture and the (generalised) Tate conjecture are usually formulated as statements for varieties over number fields or finitely generated fields. The conjectures as formulated above (for $X$ over~$\CC$) are true if and only if the analogous conjectures over arbitrary finitely generated fields (of characteristic~$0$) are true.
\end{remark}

\begin{theorem}
\label{thm:MTCTC}
Let $X$ be a complex Gushel--Mukai variety of even dimension. Then the Mumford--Tate conjecture and the generalised Tate conjecture for~$X$ are true.
\end{theorem}

To avoid any confusion: the assertion is that the mentioned conjectures are true for all~$i$ and~$\ell$. However, this is non-trivial only for the cohomology in middle degree ($i=\frac{\dim(X)}{2}$), and in the proof the choice of~$\ell$ plays no particular role. Let us further note that for K3 surfaces, the Mumford--Tate and Tate conjectures were proven in~\cite{Andre-ShafTate}.

\begin{proof}
Let $h \in H^2\left(X,\ZZ(1)\right)$ be the class of the ample generator of~$\Pic(X)$. Let $i = \frac{\dim(X)}{2}$. We need to show that $(X,h)$ satisfies conditions $\mathrm{A}_i$ and~$\mathrm{B}_i^+$ as in \cite[Section~1.4]{Andre-ShafTate}; if this is true, then the Tate conjecture and the Mumford--Tate conjecture follow from \cite[Theorems~1.5.1 and~1.6.1(4)]{Andre-ShafTate}. (Note that these results are stated over number fields, but the proofs are valid over arbitrary finitely generated base fields of characteristic~$0$, and as just remarked, this is what we need. Further, for one step in the proof, details are missing in~\cite{Andre-ShafTate}; this is corrected in~\cite[Section~2]{Moonen-TCMTC}.)
	
It is clear that condition~$\mathrm{A}_i$, which says that the middle cohomology should be of K3 type, is satisfied. The conditions~$\mathrm{B}_i^+$ state that $(X,h)$ should be a fibre in a connected family of polarised varieties such that the image of the period map contains an open subset of the period domain, and (condition~$\mathrm{B}^+_i\text{(iv)}$) such that the Hodge conjecture in middle degree is true for the fibres in this family. To see that these conditions are satisfied, fix a $5$-dimensional $\CC$-vector space~$V_5$ and abbreviate $T = \CC \oplus \wedge^2 V_5$. If $n=4$, let $S$ to be the open subscheme of $\Grass(9,T) \times \PP\left(\Sym^2(T^\vee)\right)$ consisting of the pairs $(W,Q)$ such that $\CGr(2,V_5) \cap \PP(W) \cap Q$ is a non-singular GM fourfold, and let $f \colon Y \to S$ be the tautological family of GM fourfolds. Similarly, for $n=6$ we take $S$ to be the open subscheme of quadrics $Q \in \PP\left(\Sym^2(T^\vee)\right)$ such that $\CGr(2,V_5) \cap Q$ is non-singular of dimension~$6$, and again we naturally have a family of GM varieties $f \colon Y \to S$ (\textit{cf.}\ the proof of \cite[Proposition~A.2]{KP16}.) In either case it is clear that there exists a point $0 \in S(\CC)$ such that the fibre~$Y_0$ is the GM variety~$X$ of the theorem and that the polarisation class~$H$ extends to a section of $R^2f_*\ZZ(1)$. By what is explained in \cite[Section~4]{Debarre-GM}, condition~$\mathrm{B}_i\text{(iii)}$ is satisfied, and condition~$\mathrm{B}^+_i\text{(iv)}$ follows from Theorem~\ref{thm:GHC}.
	
For the generalised Tate conjecture, finally, we have shown in the previous section (as part of the generalised Hodge conjecture) that the cohomology in middle degree (\textit{i.e.}~degree~$2i$) is supported on a subscheme of codimension~$i-1$. 
Suppose $W \subset H^{2i}\left(X,\QQ_\ell(i)\right)$ has coniveau at least~$i$. Choose a model $X_R/R$ as in Section~\ref{subsec:GTC} such that the image of the representation
\[
\rho\colon \pi_1\left(\Spec(R)\right)\tto \GL(W)
\]
is contained in the group of automorphisms of $W \cap H^{2i}\left(X,\ZZ_\ell(i)\right)$ that are the identity modulo~$\ell^3$, which is a torsion-free group (see for instance \cite[Theorems 4.5 and~5.2]{DixonEtAl}). Then by the hypothesis on the coniveau of~$W$, all eigenvalues of Frobenii on~$W$ and on~$W^\vee$ are algebraic integers, and because $W$ is pure of weight~$0$, these eigenvalues have norm~$1$ at all infinite places. Therefore, all eigenvalues of Frobenii are roots of unity by Kronecker's theorem. As the image of~$\rho$ is torsion-free and the Frobenius conjugacy classes are dense in $\pi_1\left(\Spec(R)\right)$, it follows that $W$ consists of Tate classes. By the Tate conjecture, $W$ is supported on a subscheme of codimension~$i$, and we are done.
\end{proof}

\begin{remark}
By \cite[Theorem~1.5.1]{Andre-ShafTate}, the motive (in the sense of Andr\'e) of an even-dimensional Gushel--Mukai variety~$X$ is an abelian motive; therefore, the Mumford--Tate group~$G_{\Betti}$ of $H(X,\QQ)$ equals the motivic Galois group of the Andr\'e motive $\motH(X)$. This gives a strengthening of the Mumford--Tate conjecture that is usually referred to as the ``motivated Mumford--Tate conjecture''.
	
Note that we do not know if the Chow motive of~$X$ is an abelian motive. (This is much stronger than saying that its Andr\'e motive is abelian.)
\end{remark}

\section{Chow--K\"unneth decompositions and their refinements}
\label{sec:ChowMotiveGM}

For later use, we recall some basic results on Chow--K\"unneth decompositions of (the Chow motives of) Gushel--Mukai varieties. We first work over the complex numbers; at the end of the section, we explain how to extend this to GM varieties over an arbitrary algebraically closed field of characteristic~$0$.

Let $X$ be a GM $n$-fold. The Gushel map $\gamma\colon X \to \Grass(2,V_5)$ induces a homomorphism $\gamma^* \colon H^*\left(\Grass(2,V_5),\QQ\right)$ $\to H^*(X,\QQ)$, which is surjective in all degrees different from~$n$. The cohomology of~$X$ in even degrees different from~$n$ is therefore purely of Tate type, spanned by classes of algebraic cycles, and the cohomology in odd degrees different from~$n$ is zero.

We first consider a complex GM variety~$X$ of dimension $n \in \{3,5\}$. Its intermediate Jacobian $J = J^n(X)$ (in middle degree) is a $10$-dimensional abelian variety. We refer to \cite{DK-GMJacobian} for a detailed study of~$J$.

\begin{proposition}
\label{prop:CKforGM3}
Let $X$ be a complex GM variety of dimension $n=3$ or~$5$, with intermediate Jacobian $J = J^n(X)$. Then $\frh(X)$ is an abelian Chow motive, and there is a Chow--K\"unneth decomposition
\begin{equation}
\label{eq:motGM3dec}
\text{for $n=3$:}\qquad	
\frh(X) \cong \unitmot \oplus \unitmot(-1) \oplus \frh^1(J)\left(-1\right) \oplus \unitmot(-2) \oplus \unitmot(-3),
\end{equation}
resp.\
\begin{equation}
\label{eq:motGM5dec}
\text{for $n=5$:}\qquad	
\frh(X) \cong \unitmot \oplus \unitmot(-1) \oplus \unitmot(-2)^2 \oplus \frh^1(J)\left(-2\right) \oplus \unitmot(-3)^2 \oplus \unitmot(-4) \oplus \unitmot(-5).
\end{equation}
\end{proposition}

\begin{proof}
This follows from Theorems~\ref{thm:ChowGroupGM3} and~\ref{thm:ChowGroupGM5}, together with Theorem~\ref{thm:Vial-MotDecomp}.
\end{proof}

\begin{remark}
It was shown by Laterveer, see~\cite{Laterveer-GM5}, that for $\dim(X) = 5$, there exists a decomposition as above which is multiplicative (see Section~\ref{subsec:CKDec}). For $\dim(X) = 3$, the corresponding result is not yet known.
\end{remark}

The above decompositions can be made explicit, as follows.

\subsection{}
\label{subsec:CKGM3explicit}
Let $X$ be a complex GM threefold. Let $H = -K_X \in \CH^1(X)_\QQ$ be the class of the ample generator of the Picard group, and write $\pt = \frac{1}{10}\cdot H^3$ for the class of a point on~$X$. (By Theorem~\ref{thm:ChowGroupGM3}\ref{thm:ChowGroupGM3-3}, all points on~$X$ are rationally equivalent.) We have a self-dual Chow--K\"unneth decomposition $[\Delta_X] = \sum_{i=0}^6\; \pi_X^i$ in $\CH^3(X \times X)_\QQ$, given by
\begin{align*}
\pi_X^0 &= \pt \times X,  & \pi_X^6 &= X \times \pt,\\
\pi_X^1 &= 0,   & \pi_X^5 &= 0,\\
\pi_X^2 &= \tfrac{1}{10} \cdot H^2 \times H,   & \pi_X^4 &= \tfrac{1}{10} \cdot H \times H^2,
\end{align*}
and
\[
\pi_X^3 = [\Delta_X] - \pi_X^0 - \pi_X^2 - \pi_X^4 - \pi_X^6.
\]
These projectors realise a decomposition as in~\eqref{eq:motGM3dec}; this follows from Theorem~\ref{thm:Conservative} (which applies because $\frh(X)$ is an abelian motive), where we use that the category $\Mot_{\num}(\CC)$ is semisimple (as proven by Jannsen in~\cite{Jannsen}) and that in cohomology the above projectors~$\pi_X^i$ cut out $H^i(X,\QQ)$.

\subsection{}
\label{subsec:CKGM5explicit}
Next consider a complex GM fivefold~$X$. Let $H = -\frac{1}{3}\cdot K_X \in \CH^1(X)_\QQ$ be the class of the ample generator of the Picard group, and write $\pt = \frac{1}{10}\cdot H^5$ for the class of a point on~$X$. Let $\sigma_{i,j} \in \CH^{i+j}\left(\Grass(2,V_5)\right)$ (for $3\geq i\geq j\geq 0$) be the Schubert classes.

The classes $e_1 = H^2$ and $e_2 = \gamma^*(\sigma_{1,1}) = c_2(\scrU_X)$ form a $\QQ$-basis of $\CH^2(X)_\QQ$. Define $f_1 = \frac{1}{2}\cdot H^3 - He_2$ and $f_2 = -H^3 + \frac{5}{2}\cdot He_2$ in $\CH^3(X)_\QQ$. By using the intersection matrix
\[
\begin{array}{c|cc}
& e_2 & H^2\\
\hline
H e_2 & 2 & 4\\
H^3 & 4 & 10\\
\end{array}
\]
we find that $\deg(e_i \cdot f_j) = \delta_{i,j}$. Hence we obtain a self-dual collection of mutually orthogonal projectors by setting $\pi_X^{2i-1} = 0$ for $i\neq 3$,
\begin{align*}
\pi_X^0 &= \pt \times X,  & \pi_X^{10} &= X \times \pt,\\
\pi_X^2 &= \tfrac{1}{10} \cdot H^4 \times H,  & \pi_X^8 &= \tfrac{1}{10} \cdot H \times H^4,\\
\pi_X^4 &= f_1 \times e_1 + f_2 \times e_2, & \pi_X^6 &= e_1 \times f_1 + e_2 \times f_2,
\end{align*}
and $\pi_X^5 = [\Delta_X] - \sum_{i\neq 5}\, \pi_X^i$. As before, this gives a Chow--K\"unneth decomposition.

\subsection{}
\label{subsec:CKGM46explicit}
Next we turn to GM varieties of even dimension $n \in \{4,6\}$. Let $H = -\frac{1}{n-2}\cdot K_X \in \CH^1(X)$ be the class of the ample generator of the Picard group, and write $\pt = \frac{1}{10}\cdot H^n$ for the class of a point on~$X$. (By the results in Section~\ref{sec:ChowGroupGM}, all points are rationally equivalent.) We have a self-dual Chow--K\"unneth decomposition $[\Delta_X] = \sum_{i=0}^{2n}\; \pi_X^i$ in $\CH^n(X \times X)_\QQ$ with $\pi_X^i = 0$ for all odd integers~$i$. For $n=4$, the projectors in even degree are given by
\begin{align*}
\pi_X^0 &= \pt \times X,  & \pi_X^8 &= X \times \pt,\\
\pi_X^2 &= \tfrac{1}{10} \cdot H^3 \times H,  & \pi_X^6, &= \tfrac{1}{10} \cdot H \times H^3,
\end{align*}
and
\[
\pi_X^4= [\Delta_X] - \pi_X^0 - \pi_X^2 - \pi_X^6 - \pi_X^8.
\]
These projectors realise a decomposition (for $\dim(X)=4$)
\begin{equation}
\label{eq:CKDecGM4}
\frh(X) \cong \unitmot \oplus \unitmot(-1) \oplus \frh^4(X) \oplus \unitmot(-3) \oplus \unitmot(-4).
\end{equation}

For $\dim(X) = n = 6$, the classes
\[
e_1 = H^2 ,\qquad e_2 = \gamma^*(\sigma_{1,1}) = c_2(\scrU_X)
\]
form a basis of $\CH^2(X)_\QQ$. Define the classes
\[
f_1 = \tfrac{1}{2}\cdot H^4 - H^2 e_2,\qquad f_2 = -H^4 + \tfrac{5}{2}\cdot H^2 e_2
\]
in $\CH^4(X)_\QQ$. Then $\deg(e_i \cdot f_j) = \delta_{i,j}$, and the even Chow--K\"unneth projectors are given by
\begin{align*}
\pi_X^0 &= \pt \times X,  & \pi_X^{12} &= X \times \pt,\\
\pi_X^2 &= \tfrac{1}{10} \cdot H^5 \times H,  & \pi_X^{10} &= \tfrac{1}{10} \cdot H \times H^5,\\
\pi_X^4 &= f_1 \times e_1 + f_2 \times e_2, & \pi_X^8 &= e_1 \times f_1 + e_2 \times f_2,
\end{align*}
and
\[
\pi_X^6= [\Delta_X] - \pi_X^0 - \pi_X^2 - \pi_X^4  - \pi_X^8 - \pi_X^{10}-\pi_X^{12}.
\]
These projectors realise a decomposition (for $\dim(X)=6$)
\begin{equation}
\label{eq:CKDecGM6}
\frh(X) \cong \unitmot \oplus \unitmot(-1)  \oplus \unitmot(-2)^{\oplus 2} \oplus \frh^6(X) \oplus \unitmot(-4)^{\oplus 2}  \oplus \unitmot(-5) \oplus \unitmot(-6).
\end{equation}

\subsection{}
\label{subsec:hntr}
For later use, we will also need a refinement of the Chow--K\"unneth decomposition in the even-dimensional case (\textit{cf.} \cite[Section~7.2.2]{KahnMurPedr}.)

Let $\dim(X) = n \in \{4,6\}$. By the results in Section~\ref{sec:ChowGroupGM}, the cycle class map induces an isomorphism of~$\CH^{n/2}(X)_\QQ$ with the space of Hodge classes in $H^n(X,\QQ)$. Let $\{a_1, \ldots, a_r\}$ be an orthogonal basis of $\CH^{n/2}(X)_\QQ$, and define 
\[
\pi^n_{X,\alg} = \sum_{i=1}^{r}\; \frac{1}{\deg\left(a_i^2\right)}\cdot a_i\times a_i,\qquad
\pi^n_{X,\tr} = \pi^n_X - \pi^n_{X,\alg}.
\]
These are projectors, which are independent of the chosen orthogonal basis. In cohomology, $\pi^n_{X,\alg}$ and $\pi^n_{X,\tr}$ are the projectors onto the subspace of Hodge classes in $H^4(X, \QQ)$, respectively its orthogonal complement.

Set $\frh^n_{\alg}(X) := (X,\pi^n_{X,\alg},0)$ and $\frh^n_{\tr}(X) := (X, \pi^n_{X,\tr},0)$. Then $\frh^n_{\alg}(X)\cong \unitmot(-n/2)^{\oplus r}$, and we have a refinement of \eqref{eq:CKDecGM4} and \eqref{eq:CKDecGM6} by further decomposing~$\frh^n(X)$ as 
\[
\frh^n(X) = \frh^n_{\alg}(X) \oplus \frh^n_{\tr}(X).
\]

\begin{remark}
Thus far in this section we have been working over~$\CC$. We in fact have Chow--K\"unneth decompositions as above for Gushel--Mukai varieties over an arbitrary field $K = \Kbar$ of characteristic~$0$. The best way to approach this would be to prove all results from Section~\ref{sec:ChowGroupGM} (at least with $\QQ$-coefficients) over such fields. As some of the results that we have used are documented in the literature only for complex varieties, here we  take a more direct approach.

Let $X$ be a GM variety of dimension $n \in \{3,4,5,6\}$ over~$K$. It is clear from the above explicit description of the Chow--K\"unneth projectors~$\pi^i_X$ that these same projectors are meaningfully defined over~$K$. (All we have used is the existence of the Gushel map; with a little more work, even the assumption $K=\Kbar$ could be dropped.) Taking the explicit formulas in Sections~\ref{subsec:CKGM3explicit}--\ref{subsec:CKGM46explicit} as definitions of the projectors~$\pi^i_X$, we obtain a Chow--K\"unneth decomposition $\frh(X) = \oplus_{i=0}^{2n}\; \frh^i(X)$ in~$\CHM(K)$.

These projectors are invariant under extension of the base field: if $K \subset L$ is a field extension, the Chow--K\"unneth projectors $\pi^i_{X_L} \in \CH^n(X_L \times_L X_L)$ that we obtain are the images of the projectors $\pi^i_X \in \CH^n(X\times X)$ under the natural map $\CH^n(X\times X) \to \CH^n(X_L \times_L X_L)$. Hence also $\frh^i(X_L) = \frh^i(X)_L$ for all~$i$. 
\end{remark}

\section{Motives of generalised partners and duals}

Throughout this section, $K$ is an algebraically closed field of characteristic~$0$, and $n$, $n'$ are two integers in $\{3,4,5,6\}$ of the same parity.

\begin{definition}[\textit{cf.} \protect{\cite[Definition 3.5]{KP16}}]
\label{def:Partner}
Let $X$ and $X'$ be GM varieties over~$K$ of dimensions $n$ and~$n'$, respectively. We say that $X$ and $X'$ are
\begin{itemize}
\item \emph{generalised partners} if there exists an isomorphism $V_6(X)\cong V_6(X')$ inducing an identification between the Lagrangian subspaces $A(X)\subset \wedge^3 V_6(X)$ and $A(X')\subset \wedge^3 V_6(X')$;
		
\item \emph{generalised dual} if there exists an isomorphism $V_6(X)\cong V_6(X')^\vee$ inducing an identification between the Lagrangian subspaces $A(X)\subset \wedge^3 V_6(X)$ and $A(X')^\perp\subset \wedge^3 V_6(X')^\vee$.
\end{itemize}
\end{definition}

For $n$ and~$n'$ odd, it follows from \cite[Theorem~1.1]{DK-GMJacobian} that if $X$ and~$X'$ are generalised partners or duals over~$\CC$, their middle cohomology groups are isomorphic as rational Hodge structures, up to a Tate twist. With some work, a similar conclusion can be proven for $n$ and~$n'$ even. The main result of this section is a motivic strengthening of this.

\begin{theorem}
\label{thm:MotivePartnerChar0}
Let $n$ and $n'$ be as above. Let $X$ and $X'$ be GM varieties of dimensions $n$ and~$n'$ over~$K$. If\, $X$ and~$X'$ are generalised partners or generalised duals, their rational Chow motives in middle degrees are isomorphic; i.e.
\begin{equation}
\label{eqn:IsomMotiveChar0}
\frh^n(X)\cong \frh^{n'}(X')\left(\tfrac{n'-n}{2}\right) \quad\text{in $\CHM(K)$.}
\end{equation}
\end{theorem}

Our proof will employ techniques from derived categories and is based on an idea that we learned from a draft version of the paper~\cite{BolognesiLaterveer-GM6} by Bolognesi and Laterveer.

\subsection{}
The derived category of a GM variety~$X$ admits a semi-orthogonal decomposition consisting of an exceptional collection together with an admissible subcategory $\operatorname{Ku}(X)$, called its \emph{Kuznetsov component}, which is a K3 or Enriques category depending on the parity of the dimension of the GM variety. 
More precisely,
\begin{equation}
\label{SODforGM}
\Db(X) = \left\langle \Ku(X), \scrO_X, \scrU_X^\vee, \scrO_X(1), \scrU_X^\vee(1), \dots, \scrO_X(n-3), \scrU_X^\vee(n-3) \right\rangle,
\end{equation}
where $\scrO_X(1)$ is the ample generator of the Picard group of~$X$ and $\scrU_X$ is the Gushel bundle. We denote by $i\colon \Ku(X)\hookrightarrow \Db(X)$ the natural inclusion functor and let $i^*$ and $i^!$ be, respectively, the left and right adjoints of $i$, which may be viewed as projection functors from $\Db(X)$ to the Kuznetsov component. We refer to the work of Kuznetsov--Perry~\cite{KP16} for details.

The following result is deduced from the so-called quadratic homological projective duality.

\begin{theorem}[{Kuznetsov--Perry, \textit{cf.} \cite{KuznetsovPerry-CatCone}}]
\label{thm:DerCatPartnerChar0}
Let $X$ and $X'$ be GM varieties of dimension $n$ and~$n'$ over~$K$. If\, $X$ and~$X'$ are generalised partners or generalised duals, there exists a Fourier--Mukai equivalence	
\begin{equation}
\Psi\colon \Ku(X)\xrightarrow{~\mathrm{eq}~} \Ku(X')
\end{equation}
between their Kuznetsov components.
\end{theorem}

The assertion that the equivalence is of Fourier--Mukai type means that there exists an object~$\scrE$ in $\Db(X \times X^\prime)$ such that the composition
\begin{equation}
\label{eq:PsiFM}
\Db(X) \xrightarrow{~i^*~} \Ku(X) \xrightarrow{~\Psi~} \Ku(X') \xrightarrow{~i^\prime~} \Db(X^\prime)
\end{equation}
is the Fourier--Mukai transformation~$\Phi_\scrE$ defined by~$\scrE$. This is not explicitly stated in~\cite{KuznetsovPerry-CatCone}, but it follows from \cite[Theorem~1.3]{LiPertusiZhao}.

\subsection{}
\label{subsec:A}
For the proof of Theorem~\ref{thm:MotivePartnerChar0}, we shall first work over the complex numbers. In Section~\ref{subsec:KbarGeneral} we shall explain how to deduce the result over a field $K=\Kbar$ of characteristic~$0$. The overall strategy is close to \cite{Huybrechts-MotiveK3} and uses some arguments in \cite{FuVial-K3, FuVial-Cubic4fold}.

If $Z$ is a smooth complex projective variety, an admissible subcategory of its bounded derived category of coherent sheaves $\mathcal{C}\subset \Db(Z)$ may be viewed as a non-commutative smooth proper scheme. Let $K_0(\mathcal{C})$ be the Grothendieck group of~$\mathcal{C}$. By the work of Blanc~\cite{Blanc-TopK}, we can also consider the topological K-theory $K_0^{\topo}(\mathcal{C})$ of~$\mathcal{C}$, which for $\mathcal{C} = \Db(Z)$ agrees with topological K-theory of~$Z$ (see \cite[Proposition~4.32]{Blanc-TopK}). The functors $K_0$ and~$K_0^{\topo}$ are both additive invariants, in the sense that they transform a semi-orthogonal decomposition into a direct sum decomposition. Sending an algebraic vector bundle to its underlying complex vector bundle defines a natural transformation $K_0\to K_0^{\topo}$.

We define a new invariant, which might be called the ``topologically trivial Grothendieck group'':
\begin{equation}
\A(\mathcal{C}):= \ker\left(K_0(\mathcal{C})_\QQ \to K_0^{\topo}(\mathcal{C})_\QQ\right).
\end{equation}
For a smooth complex projective variety~$Z$, we simply write~$\A(Z)$ for $\A\left(\Db(Z)\right)$. Naturally, $\A$~being the kernel of a morphism between additive functors, it is again an additive invariant. We make two basic observations:
\begin{itemize}
\item It is easy to check that $\A(\pt) = 0$. Therefore, if there is a semi-orthogonal decomposition 
\[
\Db(Z) = \left\langle \Ku(Z), E_1, \cdots, E_m \right\rangle,
\]
with $\langle E_1, \cdots, E_m\rangle$ an exceptional collection (in particular, $\langle E_i\rangle \cong \Db(\pt)$ for every~$i$), the inclusion functor $i\colon \Ku(Z)\hookrightarrow \Db(Z)$ induces an isomorphism 
\begin{equation}
\label{eqn:AofKu}
\A\left(\Ku(Z)\right)\cong \A(Z),
\end{equation}
with inverse induced by the (left or right) adjoint functor of~$i$.
 	
\item For any smooth complex projective variety~$Z$, we have a commutative diagram 
\[
\begin{tikzcd}
K_0(Z)_\QQ   \ar[r] \ar[d,"\wr", "v"'] &  K_0^{\topo}(Z)_\QQ \ar[d,"v^{\topo}","\wr"']\\
\CH^*(Z)_\QQ \ar[r, "\class"]  & H^{2*}(Z,\QQ)\rlap{,}
\end{tikzcd}
\]
where the bottom arrow is the cycle class map and the vertical isomorphisms are given by ``Mukai vector maps'' that send a class~$e$ to $\ch(e) \cdot \sqrt{\td(X)\, }$. (The right vertical map is an isomorphism by a classical result of Atiyah and Hirzebruch; see for instance \cite[Section~38.4]{FomenkoFuchs}.) {}From the above diagram, we obtain that the Mukai vector map induces an isomorphism between $\A(Z)$ and the homologically trivial part of the rational Chow group:
\begin{equation}
\label{eqn:AisomCHhom}
v\colon \A(Z) \isomlongarrow \CH^*(Z)_{\hom,\QQ}.
\end{equation} 
	
\end{itemize}
\bigskip

We can now give the proof of Theorem~\ref{thm:MotivePartnerChar0} for $X$ and~$X'$ defined over~$\CC$.

\begin{proof}[Proof of Theorem~\ref{thm:MotivePartnerChar0} over~$\CC$]
In order to avoid case distinctions, let us make the convention that when $n$ is odd, $\frh^n_{\tr}(X) = \frh^n(X)$. 

By Theorem~\ref{thm:DerCatPartnerChar0}, there exists an object $\scrE \in \Db(X\times X')$ such that the composition
\[
\Ku(X) \xhookrightarrow{~i~} \Db(X) \xrightarrow{~\Phi_{\scrE}~} \Db(X')\xrightarrow{~{i'}^*~} \Ku(X')
\]
is an equivalence of smooth proper dg-categories. (This is the same as~\eqref{eq:PsiFM} because $i^* \circ i$ is the identity on~$\Ku(X)$, and likewise for~$X^\prime$.) Applying the invariant~$\A$, we find that the composition
\begin{equation}
\label{eqn:ApplyingA}
\A\left(\Ku(X)\right)\tto \A(X) \xrightarrow{~[\scrE]_*~}  \A(X') \tto \A\left(\Ku(X')\right)
\end{equation}
is an isomorphism. However, since the Kuznetsov component of a GM variety is defined as the right orthogonal of an exceptional collection (see~\eqref{SODforGM}), \eqref{eqn:AofKu} implies that the first and  last maps in~\eqref{eqn:ApplyingA} are isomorphisms. Therefore, the middle map of \eqref{eqn:ApplyingA} is an isomorphism between $\A(X)$ and~$\A(X')$.
	
Now consider the Grothendieck--Riemann--Roch diagram
\begin{equation}
\begin{tikzcd}
\A(X) \ar[r, "{[\scrE]_*}", "\sim"'] \ar[d, "v","\wr"']& \A(X') \ar[d, "v","\wr"']\\
\CH^*(X)_{\hom, \QQ} \ar[r, "v(\scrE)_*"]& \CH^*(X')_{\hom, \QQ}\rlap{,}
\end{tikzcd}
\end{equation}
where the top map is an isomorphism as explained above and the vertical isomorphisms are the ones obtained from~\eqref{eqn:AisomCHhom}. It follows that the bottom arrow is an isomorphism. Furthermore, by our computations of Chow groups in Section~\ref{sec:ChowGroupGM}, we have that
\[
\CH(X)_{\hom, \QQ} = \CH^{\lfloor \frac{n+2}{2}\rfloor}(X)_{\hom, \QQ} = \CH\left(\frh^n_{\tr}(X)\right)_\QQ,
\]
and likewise for $X^\prime$. Now consider the morphism $v_m(\scrE) \colon \frh(X) \to \frh(X^\prime)\left(\frac{n'-n}{2}\right)$ in~$\CHM(\CC)$ defined by the component of $v(\scrE)$ in degree $m = \frac{n+n'}{2}$. It follows from the above that the morphism
\[
\pi^{n'}_{X',\tr} \circ v_m(\scrE) \circ \pi^n_{X,\tr} \colon \frh^n_{\tr}(X) \tto \frh^{n'}_{\tr}(X')\left(\tfrac{n'-n}{2}\right)
\]
induces an isomorphism on Chow groups. By Lemma~\ref{lem:ManinPrinc}, this implies that this morphism is an isomorphism.

In case $n$ and~$n^\prime$ are even, it remains to pass from $\frh^n_{\tr}(X)$ and~$\frh^{n'}_{X',\tr}$ to $\frh^n(X)$ and~$\frh^{n'}_{X'}$. By construction, there exist integers $r$ and~$r^\prime$ such that
\[
\frh^n(X) = \frh^n_{\tr}(X) \oplus \unitmot\left(-\tfrac{n}{2}\right)^{\oplus r},\qquad
\frh^{n'}(X') = \frh^{n'}_{\tr}(X') \oplus \unitmot\left(-\tfrac{n'}{2}\right)^{\oplus r'}.
\]
Taking Hodge realisations and using that the middle Betti numbers of~$X$ and~$X'$ are the same, it follows that $r=r^\prime$, and therefore $\frh^n(X)\cong \frh^{n'}(X')\left(\tfrac{n'-n}{2}\right)$.
\end{proof}

\subsection{}
\label{subsec:KbarGeneral}
To finish, we explain how to obtain Theorem~\ref{thm:MotivePartnerChar0} over an arbitrary algebraically closed field~$K$ of characteristic~$0$. Let $X$ and~$X'$ be generalised partners or duals over~$K$, of dimensions $n$, $n' \in \{3,4,5,6\}$ of the same parity. There exists a subfield $K_0 \subset K$ which is finitely generated over~$\QQ$ such that $X$ and~$X'$ both have models, say $X_0$ and~$X_0^\prime$, over~$K_0$, and such that moreover the Chow--K\"unneth projectors $\pi^n_X$ and~$\pi^{n'}_{X'}$ that cut out the submotives $\frh^n(X)$ and~$\frh^{n'}(X')$ are defined over~$K_0$. Clearly, Theorem~\ref{thm:MotivePartnerChar0} for $X_0$ and~$X_0^\prime$ over an algebraic closure of~$K_0$ implies the result over~$K$. 

Choose an embedding $K_0 \hookrightarrow \CC$, and let $\Kbar_0$ be the algebraic closure of~$K_0$ inside~$\CC$. By the now proven Theorem~\ref{thm:MotivePartnerChar0} over~$\CC$, there exists an isomorphism $\alpha \colon \frh^n(X_{0,\CC})\cong \frh^{n'}(X^\prime_{0,\CC})\left(\frac{n'-n}{2}\right)$ in~$\CHM(\CC)$. There exists a finitely generated field extension $\Kbar_0 \subset K_1$ inside~$\CC$ such that $\alpha$ and~$\alpha^{-1}$ are defined over~$K_1$. Concretely, this means there exist cycle classes $Z_1,\ldots,Z_m \in \CH\left((X_0 \times X_0^\prime)_{K_1}\right)$ and rational numbers $a_i$ and~$b_i$ such that 
\[
\alpha = \pi^{n'}_{X'_0} \circ \left(\sum a_i \cdot Z_i\right) \circ \pi^n_{X_0},\qquad
\alpha^{-1} = \pi^n_{X_0} \circ \left(\sum b_i\cdot {}^{\mathsf{t}}Z_i\right) \circ \pi^{n'}_{X'_0}.
\]

The field~$K_1$ is the function field of a variety~$S$ over~$\Kbar_0$. Choose a point $s \in S(\Kbar_0)$. By specialisation from the generic point of~$S$ to~$s$, we obtain cycle classes $Z_{1,s},\ldots,Z_{m,s} \in \CH\left((X_0 \times X_0^\prime)_{\Kbar_0}\right)$. Then 
\[
\pi^{n'}_{X'_0} \circ \left(\sum a_i \cdot Z_{i,s}\right) \circ \pi^n_{X_0}
\quad\text{and}\quad
\pi^n_{X_0} \circ \left(\sum b_i\cdot {}^{\mathsf{t}}Z_{i,s}\right) \circ \pi^{n'}_{X'_0}
\]
define mutually inverse morphisms $\frh^n(X_0)\rightleftarrows \frh^{n'}(X^\prime_0)\left(\frac{n'-n}{2}\right)$ in $\CHM(\Kbar_0)$, and this gives what we want.

%%%%%%%%%%%%%%%%%%%%%
% References
%%%%%%%%%%%%%%%%%%%%%

\newcommand{\etalchar}[1]{$^{#1}$}

\end{document}